\documentclass[12pt]{amsart}

\usepackage{amsthm}
\usepackage{amssymb,latexsym,graphics,enumerate}
\usepackage[mathscr]{eucal}
\usepackage{amsmath,amsfonts,amsthm,amssymb}
\usepackage{color}
\usepackage[left=1in,right=1in,top=1in,bottom=1in]{geometry}
\numberwithin{equation}{section}

\newcommand{\rr}{\mathbb{R}}

\newcommand{\lan}{\langle}
\newcommand{\ran}{\rangle}
\newcommand{\be}{\begin{eqnarray*}}
\newcommand{\bel}{\begin{eqnarray}}
\newcommand{\ee}{\end{eqnarray*}}
\newcommand{\eel}{\end{eqnarray}}
\newcommand{\ba}{\begin{aligned}}
\newcommand{\ea}{\end{aligned}}
\newcommand{\de}{\Delta}
\newcommand{\al}{\alpha}
\newcommand{\na}{\nabla}
\newcommand{\mc}{\mathcal}
\newcommand{\ep}{\epsilon}
\newcommand{\ra}{\rightarrow}
\newcommand{\pa}{\partial}

\newcommand{\JJ}{\mathcal{J}}

\newtheorem{thm}{Theorem}[section]

\newtheorem{defn}{Definition}[section]
\newtheorem{lem}{Lemma}[section]
\newtheorem{pro}{Proposition}[section]
\newtheorem{exm}{Example}[section]

\theoremstyle{remark}
\newtheorem{rmk}{Remark}[section]

\renewcommand{\geq}{\geqslant}

\renewcommand{\leq}{\leqslant}

\newcommand{\myr}[1]{{#1}} 

\usepackage{graphicx}
\usepackage[dvipsnames]{xcolor}
\usepackage[colorlinks=true,pdfstartview=FitV,
linkcolor=red,citecolor=ForestGreen, urlcolor=blue]{hyperref}

\begin{document}
\title[Multi-species Patlak-Keller-Segel system]{Multi-species Patlak-Keller-Segel system}

\author{Siming He}
\address{Department of Mathematics, Duke University, Durham NC 27708}
\email{siming.he@math.duke.edu}

\author{Eitan Tadmor}
\address{Department of Mathematics and Institute for Physical Sciences \& Technology (IPST)\newline
University of Maryland, College Park MD 20742}
\email{tadmor@umd.edu}

\date{\today}



\thanks{\textbf{Acknowledgment.} Research was supported in part by NSF grants DMS16-13911, RNMS11-07444 (KI-Net) and ONR grants N00014-1812465 and N00014-21-12773.}
\date{\today}

\begin{abstract}
We study the regularity and large-time behavior of a crowd of species driven by chemo-tactic interactions. What distinguishes the different  species  is the way they interact with the rest of the crowd: the collective motion is driven by different chemical reactions which end up in a \emph{coupled system} of parabolic Patlak-Keller-Segel equations.   We show that  the  densities of the different species  diffuse to zero provided the chemical interactions between the different species  satisfy certain sub-critical condition; the latter  is intimately related to a log-Hardy-Littlewood-Sobolev inequality for systems due to Shafrir \& Wolansky. Thus for example, when two species interact, one of which has mass less than $4\pi$, then the 2-system stays smooth for all time independent of the total mass of the system, in sharp contrast with the well-known breakdown  of one specie with initial mass$> 8\pi$. 
\end{abstract}

\maketitle
\setcounter{tocdepth}{1}
\tableofcontents

\section{Introduction}
In this paper, we consider the multi-species parabolic-elliptic Patlak-Keller-Segel (PKS) system which models chemotaxis phenomena involving multiple bacteria species 
\begin{align}
\left\{\begin{array}{rrr}\ba \partial_t n_\al+&\na\cdot (\na c_\al n_\al)=\de n_\al,\quad \al\in \mathcal{I},\\
-\de c_\al & =\sum_{\beta\in\mathcal{I}}b_{\al\beta}n_\beta,\\
n_\al(x,t&=0)=n_{\al 0}(x), \quad x\in \rr^2.\ea\end{array}\right.\label{EQ:KS_multiple_groups}
\end{align}
Here $n_\al,c_\al$ denote the bacteria and the chemical densities respectively. The parameters $\al,\beta\in \mathcal{I}$ indicate different species of  bacteria/chemicals. The total number of species, which is denoted $|\mathcal{I}|$ throughout the paper, is assumed to be finite. The first equation in the system \eqref{EQ:KS_multiple_groups} describes the time evolution of the bacteria density $n_\al$ subject to chemical density distribution $c_\al$ and diffusion. The second equation governs the evolution of the chemical density $c_\al$, which is determined by the collective effect of different species of bacteria $n_\beta$.  The \emph{chemical generation coefficients} $b_{\al\beta}$ represent the relative impact of the bacteria distribution $n_\beta$ on the generation of the chemical $c_\al$.

Remark that system \eqref{EQ:KS_multiple_groups} covers the more general setup, 
in which each species has its own sensitivity to the chemo-attractant, quantified by the positive constant parameters $\{\chi_\al\}$,
\begin{align}\label{eq:PKSwChi}\tag*{(\theequation)$^{\prime}$}
\left\{\begin{array}{rrr}\ba \partial_t n_\al+&\chi_\al \na\cdot (\na c_\al n_\al)=\de n_\al,\quad \al\in \mathcal{I},\\
-\de c_\al & =\sum_{\beta\in\mathcal{I}}b_{\al\beta}n_\beta,\\
n_\al(x,t&=0)=n_{\al 0}(x), \quad x\in \rr^2.\ea\end{array}\right.
\end{align}
Indeed, if we let $\eta_\al>0$ be scaling parameters at our disposal, we  set $n'_\al:=\eta_\al n_\al$ and $c'_\al:=\chi_\al c_\al$, then \ref{eq:PKSwChi} is reduced to \eqref{EQ:KS_multiple_groups} for the `tagged' variables, $(n'_\al,c'_\al)$, with re-scaled  generation array,
$b'_{\al\beta}=\chi_\al b_{\al\beta}\eta^{-1}_\beta$. In particular, choosing $\eta_\beta=1/\chi_\beta$ shows that if $\mathbf{B}=\{b_{\al\beta}\}$ is symmetric, then so is $\mathbf{B}'$.

In the last few years, social interaction within biofilms --- a special form of bacteria colonies --- has aroused increasing interest among the biology and biophysics community, \cite{CTGFD2017}. In a biofilm, billions of bacteria of different species live together and create hard-to-remove  infections. Different cells in the biofilm specialize in various tasks, acquiring food, defending  colony and preserving genetic information included. Chemical signals and ion signals are generated to communicate information within these bacteria colonies.  The multi-species PKS model \eqref{EQ:KS_multiple_groups} serves as an attempt to understand the biofilm. Moreover, in the  Chemotaxis experiment, the bacteria involved have large genetic variation. For example, E.coli only share $30\%$ of their genes. Equation \eqref{EQ:KS_multiple_groups} also serves as a more accurate model than  single species dynamics, taking into account the possible genetic variation appeared in the experiments.

We recall the large literature on the single species PKS model  \eqref{EQ:KS_multiple_groups} ($|\mathcal{I}|=1$), referring the interested reader to the review \cite{Hortsmann} and the following works \cite{BedrossianIA10}--\cite{CalvezCorriasEbde12},\cite{corrias2014existence}--\cite{CorriasPerthame08},\cite{JagerLuckhaus92},\cite{HillenPainter09},\cite{Naito06},\cite{NagaiSenbaYoshida97},
\cite{TaoWinkler12},\cite{KozonoSugiyama09}. We summarize the essential results here. The preserved total mass of the solution $M:=|n(t)|_{L^1}=|n_0|_{L^1}$ determines the long time behavior. If the intitial data $n_0$ has subcritical mass $M<8\pi$ and finite second moment, the unique global smooth solutions exist for all time, \cite{BlanchetEJDE06},\cite{CarrilloRosado10},\cite{EganaMischler16}. If $M$ is strictly greater than $8\pi$ and the second moment is finite, solution blows up in finite time, \cite{JagerLuckhaus92},\cite{Nagai95},\cite{BlanchetEJDE06}. If $M=8\pi$, solution aggregates to a Dirac mass as time tends to infinity, \cite{BlanchetCarrilloMasmoudi08}.

The multi-species PKS equation \eqref{EQ:KS_multiple_groups} has attracted increasing interest in the last decade. Its study originates in Wolansky's work \cite{Wolansky02}.
Since then, a lot of research were carried out in the specific case of two interacting  species, \cite{ConcaEspejoVilches11},\cite{BaiWinkler16},\cite{KurganovMedvidova14},\cite{ArenasStevensVelazquez09},\cite{FasanoManciniPrimicerio04},\cite{EspejoVilchesConca13}. 
Even in the two-species case, the PKS systems \eqref{EQ:KS_multiple_groups} behave differently from the single-species ones. Consider the PKS equation  \eqref{EQ:KS_multiple_groups} subject to symmetric chemical generation coefficients
\begin{align}\label{eq:twospecies}
\mathbf{B}:=\left[\begin{array}{cc}b_{11}& b_{12}\\ b_{21}&b_{22}\end{array}\right]=\left[\begin{array}{cc}
0&1\\1&0
\end{array}\right],\end{align}
which models two species with \emph{cross-attractions}. We will prove that if one species has mass strictly less than $4\pi$, the solutions to \eqref{EQ:KS_multiple_groups} exist globally regardless of the mass of the other species. However, if some critical mass constraint is violated, the solutions undergo finite time blow-up. On the other hand, for some special non-symmetric chemical generation matrices, e.g.,
\begin{align*}
\mathbf{B}=\left[\begin{array}{cc}0&1\\-2&0\end{array}  \right],
\end{align*}
the solutions $\mathbf{n}:=\{n_\al\}_{\al\in\mathcal{I}}$ to \eqref{EQ:KS_multiple_groups} decay to zero unconditionally.

In this paper, we quantify a global well-posedness  condition for the multi-species PKS model \eqref{EQ:KS_multiple_groups} subject to symmetric chemical generation coefficients, and we characterize its  long time behavior (for both --- symmetric and non-symmetric cases), along the lines of our results announced in \cite{HeThesis}.

Before stating the main theorems, we list the basic assumptions and terminologies. The following initial conditions are always assumed
\begin{equation}\label{General_condition_for_initial_data}
\sum_{\al\in \mathcal{I}}n_{\al 0}(1+|x|^2)\in L^1(\rr^2);\quad n_{\al 0}\log n_{\al 0}\in L^1(\rr^2),\quad\forall\al\in \mathcal{I}.
\end{equation}
We store the chemical generation coefficients $b_{\al\beta}$'s and the masses $M_\al=|n_\al(\cdot,t)|_1\equiv|n_{\al 0}|_1$ in compact matrix/vector form:
\begin{align}
\mathbf{B}:=\{b_{\al\beta}\}_{\al,\beta\in \mathcal I}
, \quad \mathbf{B}_+:=\{(b_{\al\beta})_+\}_{\al,\beta\in \mathcal I}, \quad \mathbf{M}:=\{M_\al\}_{\al\in \mathcal{I}}, \  M_\al=|n_{\al 0}|_1,
\end{align}
where $(\cdot)_+$ denotes the positive part of the function. We introduce the function $Q_{\mathbf{B},\mathbf{M}}$ acting on subsets $\JJ$ of the index set $\mathcal{I}$,
\begin{align}
Q_{\mathbf{B},\mathbf{M}}[\JJ]=&\frac{\sum_{\al,\beta\in \JJ}b_{\al\beta}M_\al M_\beta}{\sum_{\al\in \JJ}M_\al},\qquad \JJ\subset \mathcal{I}. \label{Q}
\end{align}
In particular, if $\JJ=\mathcal{I}$, then $Q_{\mathbf{B},\mathbf{M}}[\JJ]$ has a simple matrix representation:
$\displaystyle 
Q_{\mathbf{B},\mathbf{M}}[\mathcal I]=\frac{\lan \mathbf{B}\mathbf{M}, \mathbf{M}\ran}{|\mathbf{M}|_1}$, where $\lan\cdot, \cdot\ran,\enskip |\cdot|_1$ denote the Euclidean inner product and the $\ell^1$-vector  norm.

We first studied the multi-species PKS system \eqref{EQ:KS_multiple_groups} subject to \emph{symmetric} arrays
\begin{equation}\label{b symmetric}
b_{\al\beta}=b_{\beta\al},\quad \forall\al,\beta\in\mathcal{I}.
\end{equation}
Same as in the single species case, there exists natural dissipated free energy for the system \eqref{EQ:KS_multiple_groups}
\begin{equation}\label{Free_energy}
E[\mathbf{n}]=\sum_{\al\in \mathcal{I}}\int n_\al\log n_\al dx+\sum_{\al,\beta\in \mathcal{I}}\frac{b_{\al\beta}}{4\pi}\iint n_\al(x)\log|x-y|n_\beta(y)dxdy,\quad \mathbf{n}:=(n_\al)_{\al\in\mathcal{I}}.
\end{equation}
The proof of the dissipation of \eqref{Free_energy} is postponed to the next section. We solve the equation \eqref{EQ:KS_multiple_groups} in the distribution sense with free energy dissipation constraint. 
\begin{defn}[Free energy solutions]
For any distributional solutions $\mathbf{n}$ to the equation \eqref{EQ:KS_multiple_groups} subject to initial data $\mathbf{n}_0$, they are the free energy solutions to \eqref{EQ:KS_multiple_groups} if the following free energy dissipation inequality holds on some maximal time interval $[0,T_\star)$
\begin{equation}\label{free_energy_dissipation}
E[\mathbf{n}(t)]+\sum_{\al\in \mathcal{I}}\int_0^t\int_{\rr^2}n_\al|\na\log n_\al-\na c_{\al}|^2dxds\leq E[ \mathbf{n}_0],\quad \forall t\in [0,T_\star).
\end{equation}
If the equality in \eqref{free_energy_dissipation} is satisfied, we call it free energy dissipation equality.
\end{defn}

The existence and blow-up theorems of \eqref{EQ:KS_multiple_groups} are stated as follows.
\begin{thm}[Global existence: subcritical mass]\label{Theorem_of_subcritical_existence}
Consider the equation \eqref{EQ:KS_multiple_groups} subject to initial conditions \eqref{General_condition_for_initial_data}. If the symmetric chemical generation matrix $\mathbf{B}$ ($\mathbf{B}_+\neq \mathbf{0}$) and the mass vector $\mathbf{M}$ satisfy the following subcritical mass constraint 
\begin{subequations}\label{eqs:subcritical_mass_condition}
\begin{eqnarray}
Q_{\mathbf{B}_+,\mathbf{M}}[\mathcal{I}] &<&8\pi, \label{subcritical_mass_condition_a}\\
Q_{\mathbf{B}_+,\mathbf{M}}[\JJ]& <&Q_{\mathbf{B}_+,\mathbf{M}}[\mathcal{I}]
\ \text{ for all } \ \emptyset\neq \JJ\subsetneqq \mathcal{I}.\label{subcritical_mass_condition_b}
\end{eqnarray}
\end{subequations}
Then the free energy solutions to \eqref{EQ:KS_multiple_groups} exist for all finite time.
\end{thm}

The multi-species mass condition \eqref{eqs:subcritical_mass_condition} recovers the   threshold for global regularity of a single species (after re-scaling), $\chi M <8\pi$, which is known to be sharp \cite{JagerLuckhaus92,Nagai95,BlanchetEJDE06, CarrilloRosado10, EganaMischler16}. It also provides a sharp characterization for global regularity of two-species dynamics.\newline
 Here are \myr{three} prototypical examples.

\medskip\noindent
\begin{exm}[{\bf Competition \myr{of two species}}]
We consider  the 2-species dynamics \eqref{eq:twospecies} with general sensitivity coefficients $\chi_1,\chi_2>0$,
\[
\begin{split}
\partial_t n_1 +\chi_1\nabla\cdot&(n_1\nabla c_1) = \Delta n_1,\\
\partial_t n_2 +\chi_2\nabla\cdot&(n_2\nabla c_2) = \Delta n_2,\\
&\left\{\begin{array}{c}
-\Delta c_1 =  n_2,\\
-\Delta c_2 = n_1.
\end{array}\right.
\end{split}
\]
Theorem \ref{Theorem_of_subcritical_existence} applies to the re-scaled variables $n'_\al=n_\al/\chi_\al$ with re-scaled masses $M'_\al=M_\al/\chi_\al$ and the corresponding  re-scaled chemical generation array 
$\displaystyle 
{\mathbf B}=\left[\begin{array}{cc} 0 & \chi_1\chi_2\\ \chi_1\chi_2 & 0\end{array}\right]$. The sub-critical condition \eqref{subcritical_mass_condition_a} now reads
$({(\chi_2M_1)}^{-1}+{(\chi_1M_2)}^{-1})^{-1}<4\pi$, while   \eqref{subcritical_mass_condition_b} is void since $Q_{\mathbf{B},\mathbf{M}'}[\JJ]=0$ for $\JJ=\{1\}, \{2\}$. In particular,  if the mass of one species --- \emph{either} $\chi_2M_1$ \emph{or} $\chi_1M_2$ is strictly less than $4\pi$, then  \eqref{eqs:subcritical_mass_condition} holds: global regularity follows \underline {independently} of the mass of the other species.
\end{exm}

\myr{
\medskip\noindent
\begin{exm}[{\bf Competition \myr{of three- and many-species}}]
We consider  the 3-species dynamics \eqref{eq:twospecies} with positive sensitivity coefficients $\chi_1=\chi_3:=\chi$ and $\chi_2$,
\[
\begin{split}
\partial_t n_\al +\chi_\al\nabla\cdot(n_\al\nabla c_\al) &= \Delta n_\al, \qquad \al\in\{1,2,3\}\\
-\Delta \left[\begin{array}{c}c_1\\c_2\\c_3\end{array}\right]&= 
\left[\begin{array}{ccc} 0 & 1 & 0 \\
													 1 & 0 & 1\\
													 0 & 1 & 0 \end{array}\right] 
\left[\begin{array}{c}n_1 \\ n_2 \\ n_3\end{array}\right].
\end{split}
\]
Theorem \ref{Theorem_of_subcritical_existence} applies to the re-scaled variables $n'_\al=n_\al/\chi_\al$ with re-scaled masses $M'_\al=M_\al/\chi_\al$ and the corresponding  re-scaled chemical generation array 
$\displaystyle 
{\mathbf B}=\left[\begin{array}{ccc} 0 & \chi_1\chi_2 & 0 \\ \chi_1\chi_2 & 0 & \chi_2\chi_3\\
0 & \chi_2\chi_3 & 0 \end{array}\right]$. The sub-critical condition \eqref{subcritical_mass_condition_b} with ${\mathcal J}=\{1,2\}\subset \{1,2,3\}$  requires
\[
2\frac{M_1M_2}{M_1/\chi_1+M_2/\chi_2} <  2\frac{M_1M_2+M_2M_3}{M_1/\chi_1+M_2/\chi_2+M_3/\chi_3},
\]
which is satisfied for all $M_\al$'s (recalling that  $\chi_3=\chi_1$). Similarly, the sub-critical condition \eqref{subcritical_mass_condition_b} with ${\mathcal J}=\{2,3\}\subset \{1,2,3\}$  requires
\[
2\frac{M_2M_3}{M_2/\chi_2+M_3/\chi_3} <  2\frac{M_1M_2+M_2M_3}{M_1/\chi_1+M_2/\chi_2+M_3/\chi_3},
\]
holds  for all $M_\al$'s; finally, \eqref{subcritical_mass_condition_b} with ${\mathcal J}=\{1,3\}$ is void, and hence it remains to verify that \eqref{subcritical_mass_condition_a} holds
\[
2\frac{M_1M_2+M_2M_3}{M_1/\chi_1+M_2/\chi_2+M_3/\chi_3} < 8\pi;
\]
This inequality is satisfied if  
\[
\frac{1}{{1}/{\chi_2M_1}+{1}/{\chi_1M_2}}
+ \frac{1}{{1}/{\chi_3M_2}+{1}/{\chi_2M_3}}<4\pi
\]
For example,  if $\chi M_2 <2\pi$, then  \eqref{eqs:subcritical_mass_condition} holds: global regularity follows \underline {independently} of the mass of the other species, $M_1$ and $M_3$.\newline
This can be extended to a general  many species array
$
\left[\begin{array}{ccccc} 0 & 1 & 0 & \ldots & \ldots  \\
													   1 & 0 & 1 & 0 & \ldots \\
													   0 & 1 & \ddots & \ddots & \ddots \\
													   0 & \ddots & \ddots & 0 & 1 \\
													   0 & \ldots & \ddots & 1 & 0 \end{array}\right]
$.
\end{exm}
}

\medskip\noindent
\begin{exm}[{\bf Cooperation \myr{of two species}}]
 Consider the 2-species dynamics \cite{EspejoVilchesConca13,chertock2018high}
\[
\begin{split}
\partial_t n_1 +\chi_1\nabla\cdot(n_1\nabla c) &= \Delta n_1,\\
\partial_t n_2 +\chi_2\nabla\cdot(n_2\nabla c) &= \Delta n_2,\\
\Delta c + n_1 + n_2 - c &= 0.
\end{split}
\]
Theorem \ref{Theorem_of_subcritical_existence} applies to the re-scaled variables $n'_\al=n_\al/\chi_\al$ with re-scaled masses $M'_\al=M_\al/\chi_\al$ and the corresponding  re-scaled concentrations $c'_1:=\chi_1 c$ and $c'_2:=\chi_2 c$, coupled through the chemical generation array 
$\displaystyle 
{\mathbf B}=\left[\begin{array}{cc} \chi^2_1 & \chi_1\chi_2\\ \chi_1\chi_2 & \chi^2_2\end{array}\right]$. The sub-critical condition\eqref{eqs:subcritical_mass_condition} now reads
\[
 \max\{\chi^2_1M'_1, \chi^2_2M'_2\} < \frac{(\chi_1 M'_1+\chi_2M'_2)^2}{M'_1+M'_2}<{8}{\pi},
 \]
 or --- after scaling back,
 \begin{equation}\label{eq:exm0}
  \max\{\chi_1M_1, \chi_2M_2\} < \frac{( M_1+M_2)^2}{{M_1}/{\chi_1}+{M_2}/{\chi_2}}<{8}{\pi}.
 \end{equation}
 The inequality on the right  of \eqref{eq:exm0} coincides with the first part of characterization for global existence in \cite[Theorem 1]{EspejoVilchesConca13}. 
The  inequality on the left of \eqref{eq:exm0} holds whenever $\frac{1}{2} < {\chi_1}/{\chi_2} < 2$ (independent of the $M_i$'s). Observe that \eqref{eq:exm0} implies --- and is therefore more restrictive than the second part of the general characterization for global existence in \cite[Theorem 1]{EspejoVilchesConca13} which requires $\max\{\chi_1M_1, \chi_2M_2\}< 8\pi$. 
\end{exm}

While  the last two examples show that the sub-critical mass condition \eqref{subcritical_mass_condition_b} may or may not be sharp for general $|\mathcal{I}|\geq2$ species,  the necessity of  the upper-bound in \eqref{subcritical_mass_condition_a} is stated  in the following.

\begin{thm}[Blow-up: supercritical mass]\label{Theorem_of_supercritical_blow_up}
Consider the equations \eqref{EQ:KS_multiple_groups} subject to smooth initial data $n_\al\in H^s$, $\forall\al \in\mathcal{I}$, $s\geq 2$ with finite second moment, and governed by a symmetric  chemical generation matrix \eqref{b symmetric}. If $Q_{\mathbf{B},\mathbf{M}}[\mathcal{I}]>8\pi$, then the solution blows up at a finite time.
\end{thm}

\begin{rmk} Theorem \ref{Theorem_of_supercritical_blow_up} tells us that  the bound $Q_{\mathbf{B},\mathbf{M}}[\mathcal{I}] \leq 8\pi$ is necessary for existence of global-in-time free energy solution.  
A sufficient condition  for this (strict) bound   to hold is given by, consult Proposition \ref{prop:sufficient} below, 
\begin{equation}\label{eq:sufficient}
\rho(\mathbf{B}_+) \max_\al M_\al < 8\pi, \qquad \rho(X)_{|X\in \text{Symm}_{{\mathcal I}\times {\mathcal I}}}:=\max_\al \lambda_\al(X).
\end{equation} 
Thus,  \eqref{eq:sufficient} implies that the first inequality  \eqref{subcritical_mass_condition_a} is satisfied. 
As an example, we revisit the two-species example \eqref{eq:twospecies} (with $\chi_1=\chi_2=1$). In this case, $Q_{\mathbf{B},\mathbf{M}}[\mathcal{J}]=0$ for $\JJ\subsetneqq \mathcal{I}$, so the second inequalities in  \eqref{subcritical_mass_condition_b} 
are void: it is only the first part, \eqref{subcritical_mass_condition_a}, that needs to be verified. Here $\rho(\mathbf{B}_+)=1$ and the sufficient condition   \eqref{eq:sufficient} amounts to $\displaystyle \max_{\al\in\{1,2\}}M_\al<8\pi$, which suffices (yet  stronger than the sharp $\displaystyle (M_1^{-1}+M_2^{-1})^{-1}<4\pi$ encountered before) for \eqref{subcritical_mass_condition_a} and hence the global existence of \eqref{eq:twospecies}.
\end{rmk}

To formulate the smoothness and uniqueness theorems, we need further physical restriction on the free energy solutions. First, the physical solutions to equation \eqref{EQ:KS_multiple_groups} should satisfy the conservation of mass:\begin{subequations}
\begin{align}
|n_\al(t)|_1\equiv&|n_\al(0)|_1=M_\al,\quad \forall \al \in \mathcal{I}, \quad \forall t\in [0,T_\star).\label{Mass_conservation}
\end{align}
Moreover, by formal computation, which is postponed to the next section, we have that the total second moment of the physically relevant solutions should grow linearly
\begin{align}
V[\mathbf{n}]:= \sum_{\al\in \mathcal{I}}V_\al(t) &=\sum_{\al\in \mathcal{I}}\int n_\al(x,t)|x|^2dx \nonumber \\
 & =\bigg(\sum_\al 4 M_\al\bigg)\bigg(1-\frac{Q_{\mathbf{B},\mathbf{M}}[\mathcal{I}]}{8\pi}\bigg)t+\sum_{\al\in\mathcal{I}}V_\al(0), \quad \forall t\in[0, T_\star)\label{Total_second_moment_evolution}.
\end{align}
Finally, since it is well-known that the boundedness of the entropy $\displaystyle S[n_\al]:=\int n_\al\log n_\al$ is closely related to existence of smooth solutions, we consider free energy solutions subject to bounded entropy and free energy dissipation,
\begin{align}
\mathcal{A}_t[\mathbf{n}]:=& \sup_{s\in[0,t]}\bigg\{\sum_{\al\in\mathcal{I}}\int n_\al(x,s)\log^+n_\al(x,s) dx\bigg\}\nonumber \\
& +\sum_{\al\in\mathcal{I}}\int_0^t\int n_\al(x,s)|\na \log n_\al(x,s)-\na c_\al (x,s)|^2dxds<\infty, \quad 
 \forall t<T_\star,
\label{A_t_Bound}\end{align}\end{subequations}
where $T_\star$ denotes the maximal existing time and $\log^+$ denotes the positive part of the function $\log$. Similar quantity is defined in the paper \cite{EganaMischler16}. We say that a free energy solution is \emph{physically relevant} if it satisfies physical constraints \eqref{Mass_conservation}, \eqref{Total_second_moment_evolution} and \eqref{A_t_Bound}. Now we state the theorems concerning the smoothness, uniqueness and long-time behavior of the physically relevant free energy solutions.

\begin{thm}[Smoothnness of the free energy solutions]\label{Theorem_of_smoothness_of_solutions}
Consider the equations \eqref{EQ:KS_multiple_groups} subject to initial condition \eqref{General_condition_for_initial_data} and symmetric chemical generation matrices $\mathbf{B}$.
The physically relevant free energy solutions $(n_\al)_{\al\in \mathcal{I}}$ are smooth, i.e., $n_\al\in C^{\infty}((0, T_\star)\times\rr^2),\enskip\forall \al\in\mathcal{I}$, where $T_\star$ is the maximal existence time. Moreover, the equality holds in \eqref{free_energy_dissipation}.
\end{thm}
\begin{thm}[Uniqueness of the free energy solutions]\label{Theorem_of_uniqueness}
Consider the equation \eqref{EQ:KS_multiple_groups} subject to initial condition  \eqref{General_condition_for_initial_data} and symmetric chemical generation matrix $\mathbf{B}$
. There exists at most one physically relevant free energy solution.
\end{thm}
\begin{thm}[Long time behavior of the free energy solutions]\label{Theorem_of_long_time_decay}
Consider the solutions to \eqref{EQ:KS_multiple_groups} subject to initial condition $n_\al\in H^s,\enskip \forall \al \in \mathcal{I}, s\geq 2$ and symmetric chemical generation matrices \eqref{b symmetric}. There exists a constant $C$, which only depends on the initial data, such that the following estimate is satisfied,
\begin{equation}
\sum_{\al\in\mathcal{I}}|n_\al(t)|_2^2\leq\frac{C}{1+t},\quad \forall t\in[0,\infty).
\end{equation}
\end{thm}

If the chemical generation matrix $\mathbf{B}$ is non-symmetric, the free energy \eqref{Free_energy} defined above is no longer dissipated. As a result, we cannot use the machinery developed in \cite{BlanchetEJDE06} to prove a global well-posedness theorem. However, we can still prove the global existence and uniform in time boundedness results for the multi-species PKS systems \eqref{EQ:KS_multiple_groups} subject to a special class of chemical generation matrices which we call \emph{essentially dissipative matrices}. The definition is as follows:
\begin{defn}
Define the sequences of subsets $\mathcal{I}^{(0)}\subset \mathcal{I}^{(1)}\subset...\subset \mathcal{I}^{(|\mathcal{I}|)}$ of $\mathcal{I}$ as follows:
\begin{align*}
&\mathcal{I}^{(0)}:=\{\al\in\mathcal{I}|b_{\al\beta}\leq 0,\quad \forall \beta \in \mathcal{I}\};\\
&\mathcal{I}^{(k)}:=\{\al\in\mathcal{I}|b_{\al\beta}\leq 0,\quad \forall \beta \in \mathcal{I}\backslash \mathcal I^{(k-1)}\},\quad k\in\{1,2,...,|\mathcal{I}|\}.
\end{align*}
If $\mathcal{I}^{(|\mathcal{I}|)}=\mathcal{I}$, we called the matrix $\mathbf{B}$ essentially dissipative.
\end{defn}
\begin{rmk}
The simplest essentially dissipative matrices $\mathbf{B}$'s are
\begin{align*}
\left[\begin{array}{cc}
0&1\\-1&0
\end{array}\right],\quad\left[\begin{array}{ccc}
0&1&2\\-1&0&3\\-2&-4&0
\end{array}\right].\end{align*}
Essentiall dissipative matrices naturally arise when there are  chasing-escaping phenomena in the multi-species PKS system \eqref{EQ:KS_multiple_groups}. For example, the system \eqref{EQ:KS_multiple_groups} subject to chemical generation relation $b_{12}=-b_{21}=1$, $b_{11}=b_{22}=0$ describes the situation that bacteria of species $1$ are escaping from bacteria of species $2$, whereas bacteria of species $2$ are chasing bacteria of species $1$.
\end{rmk}
The theorem corresponding to the multi-species PKS model \eqref{EQ:KS_multiple_groups} subject to essentially dissipative $\mathbf{B}$ is as follows.
\begin{thm}[Non-symmetric  interactions]\label{Theorem_essentially_negative_B}
Consider the multi-species PKS system \eqref{EQ:KS_multiple_groups} subject to initial condition $(n_\al)_{0}\in H^s$, $\forall\al\in\mathcal{I}, s\geq 2$. Assume that the chemical generation matrix $\mathbf{B}$ is essentially dissipative. Then there exists a uniformly bounded $H^s$ solution to the equation \eqref{EQ:KS_multiple_groups} for all time, i.e., there exists a constant $C_{H^s}=C_{H^s}(\{n_{\al 0}\}_{\al\in\mathcal{I}})$ such that
\begin{align*}
\sum_{\al\in\mathcal{I}}|n_\al(t)|_{H^s}\leq C_{H^s}<\infty,
\quad\forall t\in[0,\infty).
\end{align*}
Furthermore, there exists a constant $C$, which depends only on the initial data and $\mathbf{B}$, such that the following estimate is satisfied,
\begin{equation}
\sum_{\al\in\mathcal{I}}|n_\al(t)|_2^2\leq \frac{C}{1+t}, \quad \forall t\geq0.
\end{equation}
\end{thm}

The paper is organized as follows: in section 2, we give preliminaries and the proof of Theorem \ref{Theorem_of_supercritical_blow_up}; in section 3, we prove the existence of global free energy solutions with subcritical mass; in section 4, we prove the smoothness of the free energy solutions; in section 5, we prove the uniqueness of the free energy solutions; in section 6, we explore the long-time behavior of the free energy solutions; in the last section, we discuss the non-symmetric case.
\subsection{Notations}
In the paper, we use the notation $A\lesssim B$ $(A,B\geq 0)$, if there exists a constant $C$ such that $A\leq CB$. We will also use $\sum_\al$ to represent $\sum_{\al\in\mathcal{I}}$ unless otherwise stated. Constant $C_S$, $C_{HLS}$, $C_{lHLS}$, $C_{GNS}$ and $C_N$ are used to represent universal constant depending on various differential(integral) inequalities. The exact values might change from line to line. Given a vector $\mathbf{w}$ we let $|\mathbf{w}|_p$
denote its $\ell^p$ norm; given a vector function $\mathbf{w}(\cdot)$ we let $|\mathbf{w}(\cdot)|_X$ denote its norm in vector space $X$. In particular, $|\mathbf{w}(\cdot)|_p$ denote the usual $L^p$  spaces, and  the distinction between $\ell^p$ and $L^p$ spaces is clear from the text.

\section{Preliminaries}
Two quantities are crucial in the analysis of the multi-species PKS dynamics \eqref{EQ:KS_multiple_groups} --- the free energy $E[\mathbf{n}]$ \eqref{Free_energy} and the second moment $\sum_\al V_\al$ \eqref{Total_second_moment_evolution}. In this section, we calculate the time evolution of these two quantities formally and give the proof of Theorem \ref{Theorem_of_supercritical_blow_up}.

Same as in the single species case, the free energy $E[\mathbf{n}]$ \eqref{Free_energy} is formally dissipated under the equation \eqref{EQ:KS_multiple_groups}.
\begin{lem}
Consider smooth solutions $\mathbf{n}$ to the equation \eqref{EQ:KS_multiple_groups} subject to initial data $\mathbf{n}_0$ and symmetric $\mathbf{B}$, the free energy $E[\mathbf{n}]$ \eqref{Free_energy} is deceasing and it satisfies the following free energy dissipation equality 
\begin{equation}\label{Free_energy_dissipation}
E[\mathbf{n}(t)]=E[\mathbf{n}_0]-\sum_{\al\in\mathcal{I}}\int_0^t\int n_\al|\na \log n_\al -\na{c_\al}|^2dxds=:E[\mathbf{n}_0]-\int_0^t \mathcal{D}[\mathbf{n}(s)]ds.
\end{equation}
\end{lem}
\begin{proof} We apply the equation \eqref{EQ:KS_multiple_groups} and the symmetric condition \eqref{b symmetric} to calculate the time evolution of the free energy $E[\mathbf{n}]$
\begin{align}
\frac{d}{dt}E[\mathbf{n}]=&\sum_\al\int (n_\al)_t\log n_\al -\sum_{\al}\int \frac{c_\al(n_\al)_t}{2}dx-\sum_{\al}\int \frac{(c_\al)_t n_\al}{2}dx\nonumber\\
=&\sum_\al\int (n_\al)_t\log n_\al -\sum_{\al}\int \frac{c_\al(n_\al)_t}{2}dx+\sum_{\al,\beta}\frac{b_{\al\beta}}{4\pi}\int (n_\beta)_t(y) \log |x-y| n_\al(x) dxdy\nonumber\\
=&\sum_\al\int (n_\al)_t\log n_\al -\sum_{\al}\int \frac{c_\al(n_\al)_t}{2}dx+\sum_{\al,\beta}\frac{b_{\al\beta}}{4\pi}\int (n_\al)_t(x) \log |x-y| n_\beta(y) dxdy\nonumber\\
=&\sum_{\al}\int (n_\al)_t(\log n_\al - {c_\al})dx.\label{Free_energy_dissipation_1}
\end{align}
Since the equation \eqref{EQ:KS_multiple_groups} can be rewritten as
\begin{align*}
\partial_t n_\al=\na\cdot(n_\al (\na \log n_\al-\na c_\al)),
\end{align*}
applying integration by parts on the time evolution of $E[\mathbf{n}]$ \eqref{Free_energy_dissipation_1} yields
\begin{align*}
\frac{d}{dt}E[\mathbf{n}]=-\sum_{\al}\int n_\al|\na \log n_\al - \na{c_\al}|^2dx\leq 0.
\end{align*}
Now by integration in time, we obtain \eqref{Free_energy_dissipation}.
\end{proof}

Next we give the time evolution of the second moment.
\begin{lem}
Consider the smooth solutions $\mathbf{n}$ to the equation \eqref{EQ:KS_multiple_groups} subject to smooth initial data $\mathbf{n}_0\in H^s$, $s\geq 2$ and symmetric chemical generation matrix $\mathbf{B}$. The time evolution of the total second moment $\sum_{\al\in\mathcal{I}} V_\al$ \eqref{Total_second_moment_evolution}
satisfies the following equality
\begin{align}\label{time_evol_V}
\frac{d}{dt}V[\mathbf{n}]=\frac{d}{dt}\sum_{\al\in\mathcal{I}} V_\al=\bigg(\sum_{\al\in\mathcal{I}} 4 M_\al\bigg)\bigg(1-\frac{Q_{\mathbf{B},\mathbf{M}}[\mathcal{I}]}{8\pi}\bigg),
\end{align}
where $Q_{\mathbf{B},\mathbf{M}}$ is defined in \eqref{Q}.
\end{lem}
\begin{proof}
Applying the equation \eqref{EQ:KS_multiple_groups}, the definition of $Q_{\mathbf{B},\mathbf{M}}$ \eqref{Q} and the symmetry condition \eqref{b symmetric}, we calculate the time evolution of the total second moment as follows
\begin{align*}
\frac{d}{dt}\sum_{\al} V_\al
=&\sum_\al 4 M_\al+\sum_{\al} \int 2x \cdot (\na c_\al n_\al)dx\\
=&\sum_\al 4 M_\al-\sum_{\al,\beta} b_{\al\beta}\frac{1}{2\pi}\iint \frac{2x \cdot (x-y)}{|x-y|^2} n_\beta(y) n_\al(x)dxdy\\
=&\sum_\al 4 M_\al-\sum_{\al,\beta} b_{\al\beta}\frac{1}{4\pi}\iint \frac{2(x -y)\cdot (x-y)}{|x-y|^2} n_\beta(y) n_\al(x)dxdy\\
=&\sum_\al 4 M_\al-\sum_{\al,\beta} b_{\al\beta}\frac{M_\al M_\beta}{2\pi}\\
=&\bigg(\sum_\al 4 M_\al\bigg)\bigg(1-\frac{Q_{\mathbf{B},\mathbf{M}}[\mathcal{I}]}{8\pi}\bigg).
\end{align*}
This completes the proof of the lemma.
\end{proof}
\begin{rmk}
Note that in the proofs of these two lemmas, the symmetry of the matrix $\mathbf{B}$ is always assumed. In the non-symmetric case, i.e., $b_{\al\beta}\neq b_{\beta\al}$, neither of these lemmas can be applied. This is the main difficulty we faced when applying the free energy machinery in the non-symmetric case.
\end{rmk}
\begin{proof}[Proof of Theorem \ref{Theorem_of_supercritical_blow_up}]
Suppose that the solution $\mathbf{n}$ is smooth for all time. By the assumption $Q_{\mathbf{B},\mathbf{M}}[\mathcal{I}]>8\pi$, we have that the time evolution \eqref{time_evol_V} is a strictly negative constant. As a result, the total second moment will decrease to zero at a finite time $T_\star$ while the $L^1$ norm of the solution $\sum_{\al\in \mathcal{I}}|n_\al|_1$ is preserved. At time $T_\star$, the smoothness assumption of the solution will be  contradicted. Hence the solution must lose $H^s$ regularity before $T_\star$.
\end{proof}
\section{Global existence for subcritical data}
\subsection{A priori estimate on entropy}
In the case of a single species, the analysis of  PKS equation proceeds by  combining an a priori estimate of the free energy \eqref{free_energy_dissipation} together with a logarithmic Hardy-Littlewood-Sobolev inequality to recover a uniform in time a priori bound on the entropy, which in turn yields existence of free energy solution for all time. In the present context of a \emph{coupled} system of PKS equations, one seeks the corresponding log-Hardy-Littlewood-Sobolev inequality for systems which guarantees a finite lower bound of the multi-species functional $\Psi[\mathbf{n}], \ 
\mathbf{n}:=\{n_\al\}_{\al\in {\mathcal I}}$,
\begin{equation}\label{Psi R2}
\Psi[\mathbf{n}]:=\sum_{\al\in \mathcal{I}}\int_{\rr^2} n_\al\log n_\al dx+\frac{1}{4\pi}\sum_{\al,\beta\in \mathcal{I}}a_{\al\beta}\iint_{\rr^2\times \rr^2} n_\al(x)\log|x-y|n_\beta(y)dxdy,
\end{equation}
overall $n_\al$'s in the function space
\begin{equation}\label{Gamma M}
\begin{split}
 \Gamma_\mathbf{M}(\rr^2)=\Big\{(n_\al)_{\al\in \mathcal I} &\ \ n_\al\geq 0, \ \Big| \    \int_{\rr^2}n_\al|\log n_\al|dx<\infty, \\
 &  \int_{\rr^2} n_\al dx=M_\al, \quad  \int_{\rr^2} n_\al\log (1+|x|^2)dx<\infty,\forall \al\in \mathcal{I}\Big\}.
 \end{split}
\end{equation}
To this end we follow \cite{SW05}. For an arbitrary  subset of our index set, $\mathcal{J}\subset \mathcal{I}$,  one  defines the  quantity,
\begin{equation}\label{Lambda J}
\Lambda_\JJ(\mathbf{M}):=8\pi\sum_{\al\in \JJ}M_\al-\sum_{\al,\beta\in \JJ}a_{\al\beta}M_\al M_\beta,\qquad\mathbf{M}:=(M_\al)_{\al\in \mathcal{I}},\quad |\mathcal{I}|<\infty.
\end{equation}
\begin{thm}[{\cite[Theorem 4]{SW05}}]\label{thm:log HLS for systems}
Let $\mathbf{A}=(a_{\al\beta})_{\al,\beta\in \mathcal{I}}$ be  a symmetric matrix with positive entries $a_{\al\beta}\geq0$.

a) The following
\begin{align}\label{general_condition}
\left\{\begin{array}{rrr}\ba
&\Lambda_{\mathcal{I}}(\mathbf{M})=0,\\
&\Lambda_\JJ(\mathbf{M})\geq 0, \quad\forall \emptyset\neq \JJ\subset \mathcal{I},\\
&\text{if }\Lambda_\JJ(\mathbf{M})=0\text{ for some }\JJ, \text{ then }a_{\al\al}+\Lambda_{\JJ\backslash \{\al\}}(\mathbf{M})>0,\quad \forall \al\in \JJ,\ea\end{array}\right.
\end{align}
is a necessary and sufficient condition for the lower-bound of the PKS functional $\displaystyle \min_{\mathbf{n}\in \Gamma_\mathbf{M}(\rr^2)}\Psi[\mathbf{n}]$;

b) Moreover, the functional  $\Psi[\mathbf{n}]$ admits a minimizer over  $\Gamma_{\mathbf{M}}(\rr^2)$ if and only if $\Lambda_{\mathcal{I}}(\mathbf{M})=0$ and $\Lambda_\JJ(\mathbf{M})> 0$ for any $\emptyset\neq \mathcal{J}\subsetneqq \mathcal{I}$. In this case, there exists a constant, $C=C_{lHLS}$ depending on $\mathbf{M}$ and $\mathbf{B}=\{b_{\al\beta}\}$, such that the following holds
\begin{align}
\Psi[\mathbf{n}]\geq -C_{lHLS}\left(\mathbf{M},\mathbf{B}\right).\label{log_HLS_for_system}
\end{align}
\end{thm}
\begin{rmk}
As noted in \cite[p. 414]{SW05}, if the condition $\Lambda_{\mathcal{J}}\geq0$ is violated for some $\emptyset\neq \mathcal{J}\subsetneqq \mathcal{I}$, then  a scaling argument yields that the functional $\Psi[\mathbf{n}]$ on the sphere $S^2$ has no lower bound. One might be able to use this property to construct blow-up solutions \emph{on the plane}, when the following strict monotonicity  fails (recalling the functional $Q_{\mathbf{B}_+,\mathbf{M}}$ in  \eqref{Q}
\[
Q_{\mathbf{B}_+, \mathbf{M}}(\mathcal{J})<Q_{\mathbf{B}_+,\mathbf{M}}(\mathcal{I}) \ \mbox{ for all } \  \mathcal{J}\subsetneqq \mathcal{I}.
\]
\end{rmk}
The above theorem yields the following.
\begin{pro}\label{Prop_1}
Consider the equation \eqref{EQ:KS_multiple_groups} subject to smooth initial data and chemical generation coefficient matrix $\mathbf{B}$.  Further assume that $\mathbf{B}_+$ is not a zero matrix. Suppose that \eqref{eqs:subcritical_mass_condition} holds ,
\[
Q_{\mathbf{B}_+,\mathbf{M}}[\JJ]< Q_{\mathbf{B}_+,\mathbf{M}}[\mathcal{I}]<8\pi, \qquad \emptyset\neq\JJ\subsetneqq \mathcal{I},
\]
 then the total entropy $\sum_\al\int n_\al\log n_\al dx$ is bounded for all finite time.
\end{pro}
\begin{rmk}
We will not lose generality if we assume that $\mathbf{B}_+$ is not a zero matrix. If all the entries in $\mathbf{B}$ is negative, classical techniques are sufficient to analyze the system.
\end{rmk}

\begin{proof}First we rewrite the free energy dissipation relation \eqref{Free_energy_dissipation} as follows
\begin{align}
E[\mathbf{n}_{0}]\geq E[\mathbf{n}]\geq&\sum_{\al\in \mathcal{I}}\int n_\al\log n_\al dx+\sum_{\al,\beta\in \mathcal{I}}\frac{(b_{\al\beta})_+}{4\pi}\iint n_\al(x)\log|x-y|n_\beta(y)dxdy\nonumber\\
&-\sum_{\al,\beta\in\mathcal{I}}\frac{(b_{\al\beta})_-}{4\pi}\iint_{|x-y|\geq 1}n_\al(x)\log|x-y|n_\beta(y)dxdy\nonumber\\
=&(1-\theta)\sum_{\al\in\mathcal{I}}\int n_\al\log n_\al dx\nonumber\\
 &+\theta\left(\sum_{\al\in\mathcal{I}} \int n_\al\log n_\al dx+\frac{1}{4\pi}\sum_{\al,\beta\in \mathcal{I}}\frac{(b_{\al\beta})_+}{\theta}\iint n_\al(x)\log |x-y| n_\beta(y)dxdy\right)\nonumber\\
&-\sum_{\al,\beta\in\mc{I}}\frac{(b_{\al\beta})_-}{4\pi}(M_\al V_\beta+M_\beta V_\al).\label{En_lower_bound_1}
\end{align}
Define $a_{\al\beta}:=(b_{\al\beta})_+/\theta\geq 0$, $0<\theta<1$.

In order to apply Theorem \ref{thm:log HLS for systems}, we need to check the condition \eqref{general_condition}. By choosing $\theta$ properly, we make sure that the first condition $\Lambda_{\mathcal{I}}(\mathbf{M})=0$ in \eqref{general_condition} is satisfied. Direct calculation yields that
\begin{align*}
\Lambda_{\mc{I}}(\mathbf{M})=0
\Leftrightarrow&\theta=\frac{\sum_{\al,\beta\in \mathcal{I}}{(b_{\al\beta})_+}M_\al M_\beta}{8\pi\sum_{\beta\in \mathcal{I}}M_\beta}=\frac{Q_{\mathbf{B}_+,\mathbf{M}}[\mathcal{I}]}{8\pi}.
\end{align*}
Note that the assumption $Q_{\mathbf{B}_+,\mathbf{M}}(\mc{I})<8\pi$ guarantees that $\theta<1$. Next we check the remaining conditions in  \eqref{general_condition}. Recalling the definition of $\theta$ and $Q_{\mathbf{B}_+,\mathbf{M}}[\JJ]$, the following condition guarantees the existence of the minimizer of  $\Psi$ in $\Gamma_{\mathbf{M}}(\rr^2)$
\begin{align*}
Q_{\mathbf{B}_+,\mathbf{M}}[\mathcal{I}]&> Q_{\mathbf{B}_+,\mathbf{M}}[\JJ],\quad\forall\emptyset\neq \JJ\subsetneqq \mathcal{I},\\
\Leftrightarrow&\Lambda_\JJ(\mathbf{M})=8\pi\sum_{\beta\in \JJ}M_\beta-\frac{{8\pi\sum_{\beta\in \mathcal{I}}M_\beta}}{\sum_{\al,\beta\in \mathcal{I}}{(b_{\al\beta})_+}M_\al M_\beta}\sum_{\al,\beta\in \JJ}(b_{\al\beta})_+M_\al M_\beta> 0,\quad\forall\emptyset\neq \JJ\subsetneqq \mathcal{I},\\
\Leftrightarrow
&\Lambda_\JJ(\mathbf{M})> 0,\quad\forall\emptyset\neq \JJ\subsetneq \mathcal{I}.
\end{align*}
Now combining Theorem \ref{thm:log HLS for systems}, the boundedness of the second moment \eqref{time_evol_V} and the fact that $0<\theta<1$ yields that
\begin{align*}
 E[\mathbf{n}_{0}]\geq&E[\mathbf{n}]\geq(1-\theta)\sum_{\al\in\mathcal{I}}\int n_\al\log n_\al-\theta C_{lHLS}-\frac{1}{4\pi}\sum_{\al,\beta\in\mc I}(b_{\al\beta})_-(M_\al V_\beta+M_\beta V_\al),\\
\Rightarrow&\sum_{\al\in \mathcal{I}}\int n_\al\log n_\al dx\leq \frac{E[\mathbf{n}_0]+\theta C_{lHLS}+\frac{1}{2\pi}\sum_{\al,\beta}(b_{\al\beta})_-M_\al V_\beta}{1-\theta}<\infty.
\end{align*}
This completes the proof.
\end{proof}

The proof above  shows  that the log-HLS will not hold if $supp(\textbf{B})\subsetneqq\mathcal{I}$, or else we can choose  $\mathcal{J}=supp(\mathbf{B})\subsetneqq \mathcal{I}$ for which
\begin{align*}
\Lambda_\JJ(\mathbf{M})=8\pi\sum_{\beta\in \JJ}M_\beta-\frac{{8\pi\sum_{\beta\in \mathcal{I}}M_\beta}}{\sum_{\al,\beta\in \mathcal{I}}{(b_{\al\beta})_+}M_\al M_\beta}\sum_{\al,\beta\in \JJ}(b_{\al\beta})_+M_\al M_\beta< 0.
\end{align*}
The precise characterization of $\mathbf{B}$'s  such that  \eqref{eqs:subcritical_mass_condition} holds remains open; consult our conjecture in remark \ref{rmk:conjecture} below. 

The precise characterization of $\mathbf{B}$'s  such that both conditions \eqref{eqs:subcritical_mass_condition} hold remains open. We prove below the a sufficient condition, claimed in \eqref{eq:sufficient},  for the upper-bound \eqref{subcritical_mass_condition_a} to hold.
\begin{pro}\label{prop:sufficient}
Let $\mathbf{A}=(a_{\al\beta})_{\al,\beta\in \mathcal{I}}$ be  a symmetric matrix with positive entries $a_{\al\beta}\geq0$, then 
\[
Q_{{\mathbf A},M}[{\mathcal I}] < \rho(\mathbf{A})\max_\al M_\al.
\]
\end{pro}
To verify \eqref{general_condition}, we express $\mathbf{A}$ in terms of its spectral decomposition $\mathbf{A}=\sum_\al \lambda_\al \mathbf{w}_\al\mathbf{w}^*_\al$ where $\{(\lambda_\al,\mathbf{w}_\al)\}$ are the ortho-normal eigensystem of $\mathbf{A}$. We compute
\[
\langle {\mathbf A} {\mathbf M},{\mathbf M}\rangle = \sum_\al \lambda_\al |\langle \mathbf{M},\mathbf{w}_\al\rangle|^2 \leq \max_\al \lambda_\al |{\mathbf M}|_2^2 \leq \max_\al \lambda_\al|{\mathbf M}|_1 \max_\al M_\al
\]
and the result follows, $Q_{{\mathbf A},{\mathbf M}}[{\mathcal I}] \leq \rho({\mathbf A}) \max_\al M_\al$.
\subsection{Local existence and extension theorems}
Before introducing the local existence theorems of the free energy solutions, we regularize the system \eqref{EQ:KS_multiple_groups} by appropriately truncating the singularity in the convolution kernel $\na K=\na(-\de)^{-1}$:
\begin{align*}
K^\ep(z):=&K^1\bigg(\frac{|z|}{\ep}\bigg)-\frac{1}{2\pi}\log\ep;\\
K^1(|z|):=&-\frac{1}{2\pi}\log |z|,\quad |z|\geq 4,\\
K^1(|z|):=&0,\quad |z|\leq 1
\end{align*}
to get the following regularized multi-species PKS system
\begin{subequations}\label{EQ:KS_multiple_groups_regularized}
\begin{align}
\partial_t n_\al^\ep +\na\cdot (\na c_\al^\ep n_\al^\ep)&=\de n_\al^\ep,\\
 c_\al^\ep & =K^\ep*\bigg(\sum_{\beta\in\mathcal{I}}b_{\al\beta}n_\beta\bigg),\\
n_{\al}^\ep(t=0) &=\min\{n_{\al 0},\ep^{-1}\} ,\quad \forall\al\in \mc I, \quad x\in \rr^2.
\end{align}\end{subequations}
Note that the masses of the solutions $M_\al=|n_\al|_1$ are preserved in time.

Since $|\na K^\ep|_\infty$ is bounded for any fixed positive $\ep$, applying the Young's convolution inequality yields that the vector field $\na c_\al$ is bounded in $L^\infty$, i.e., $\displaystyle{\sum_\al|\na c_\al|_\infty\leq \sum_{\al,\beta} |\na K^\ep|_\infty |b_{\al\beta}|M_\beta}$. Now standard convection-diffusion PDE theory can be applied to show that the regularized system \eqref{EQ:KS_multiple_groups_regularized} admits global solutions in $L^2((0,T]; H^1)\cap C((0,T]; L^2)$. 

The following two propositions are the main local existence theorems.
\begin{pro}\label{pro 2.1 multiple group}(Criterion for Local Existence)
Let $(n_\al^\ep)_{\al\in\mathcal{I}}$ be the solutions to the regularized multi-species PKS system \eqref{EQ:KS_multiple_groups_regularized} on the time interval $[0,T)$ subject to initial constrain \eqref{General_condition_for_initial_data}. If the total entropy $\sum_\al S[n_\al^\ep]$ is bounded from above uniformly in $\ep$, i.e.,
\begin{align}
\sum_{\al\in\mathcal{I}}S[n^\ep_\al(t)]=\sum_{\al\in\mathcal{I}}\int n^\ep_\al(x,t)\log n^\ep_\al(x,t) dx\leq C_{L\log L}<\infty, \quad \forall t\in [0,T],
\end{align}
then there exists a subsequence of $\{(n_\al^\ep)_{\al \in \mathcal{I}}\}_{\ep>0}$ converging in the $L_t^2L_x^2$ strong topology to a non-negative free-energy solution to the multi-species PKS system \eqref{EQ:KS_multiple_groups} subject to initial data $({n_\al})_{0}$ on the time interval $[0,T]$. 
\end{pro}

\begin{pro}\label{pro 1.2 multiple group}(Blow-up Criterion of Free-energy Solutions) Consider the multi-species PKS system \eqref{EQ:KS_multiple_groups} subject to initial condition \eqref{General_condition_for_initial_data}. There exists a maximal existence time $T^*>0$ for the free-energy solution to the system \eqref{EQ:KS_multiple_groups}. Moreover, if $T^*<\infty$, then there exists an $\al\in \mathcal{I}$ such that
\be
\lim_{t\rightarrow T^*}\int_{\rr^2}n_\al(t)\log n_\al(t) dx=\infty.
\ee
\end{pro}

\begin{proof}[Proof of proposition \ref{pro 2.1 multiple group}]
The proof is divided into three main steps.

\medskip\noindent
\textbf{$\bullet$ STEP \#1}. Here we prove A priori estimates on mass distribution $n^\ep$ and chemical distribution $c_\al^\ep$ to prepare for the latter steps. For the readers' convenience, we summarize the uniform in $\ep$ estimates we obtained in this step:
\begin{subequations}
\begin{align}
\sum_\al|(1+|x|^2)n_\al^\ep&|_{L_t^\infty(0,T;L_x^1)}\leq C_V(\{(V_\al)_0\}_{\al\in\mathcal{I}},\mathbf{M})<\infty;\label{A_priori_est_of_n_c_al_ep_1}\\
\sum_\al|n_\al^\ep\log^\ep n_\al^\ep & |_{L_t^\infty(0,T; L^1_x)}\leq C(C_{L\log L}, C_V)<\infty;\label{A_priori_est_of_n_c_al_ep_2}\\
\sum_\al|\na \sqrt{ n_\al}|& _{L^2_t(0,T;L^2_x)}^2\leq C(C_{L\log L}, C_V)<\infty;\label{A_priori_est_of_n_c_al_ep_3}\\
\sum_{\al}|\sqrt{n_\al}\na c_\al|
&_{L^2_t(0,T;L^2_x)}^2\leq C(C_{L\log L}, C_V)<\infty.\label{A_priori_est_of_n_c_al_ep_4}
\end{align}
\end{subequations}
Before proving these estimates, we recall the following Gagliardo-Nirenberg-Sobolev inequality, which is applied several times in the sequel:
\begin{align}
|u|_{L^p}^2\leq C_{GNS}|\na u|_{L^2}^{2-4/p}|u|_{L^2}^{4/p}, \forall u\in H^1, \forall p\in [2,\infty).\label{GNS_Lp}
\end{align}

We start by proving the second moment control of the solutions \eqref{A_priori_est_of_n_c_al_ep_1}. 
Similar to the calculation in the proof of Lemma \ref{Second_moment_control}, we have the following:
\bel
\frac{d}{dt}\bigg(\sum_\al \int n_\al|x|^2dx\bigg) \leq 4\sum_\al M_\al+\sum_{\al,\beta}(b_{\al\beta})_-\frac{M_\al M_\beta}{2\pi},\label{Second_moment_control}
\eel
from which the estimate \eqref{A_priori_est_of_n_c_al_ep_1} follows directly.

To prove the $L^1$ control of $n_\al^\ep\log n^\ep_\al$ \eqref{A_priori_est_of_n_c_al_ep_2}, we recall the following lemma.
\begin{lem}
For any $g$ such that $(1+|x|^2)g\in L^1_+(\rr^2)$, we have $g\log^- g\in L^1(\rr^2)$ and
\bel
\int_{\rr^2}g\log^-g dx\leq \frac{1}{2}\int_{\rr^2}g(x)|x|^2dx+\log(2\pi)\int_{\rr^2}g(x)dx+\frac{1}{e}.\label{negative_part_of_entropy}
\eel
\end{lem}
\begin{proof}
The proof of the lemma can be found in the paper \cite{BlanchetEJDE06} and \cite{BlanchetCarrilloMasmoudi08}. We refer the interested readers to these papers for further details.
\end{proof}
The estimate \eqref{negative_part_of_entropy} yields that
\begin{align*}
\int |n_\al^\ep\log n_\al^\ep | dx & \leq \int n_\al^\ep(\log n_\al^\ep +|x|^2)dx+2\log(2\pi)M_\al+\frac{2}{e}\\
& \leq C_{L\log L}+C_V+2\log(2\pi)M_\al+\frac{2}{e}.
\end{align*}
As a result, we prove \eqref{A_priori_est_of_n_c_al_ep_2}.

Next we show the bound of $|\na \sqrt{n_\al}|^2_{L_t^2(0,T;L^2_x)}$ \eqref{A_priori_est_of_n_c_al_ep_3}. This term naturally arises when we calculate the time evolution of the entropy $\sum_\al S[n_\al]$:
\begin{align}\label{S_n_al_time_evolution}
\frac{d}{dt}\sum_{\al\in \mathcal{I}} S[n_\al]=&-4\sum_{\al\in\mathcal{I}}\int |\na\sqrt{n_\al }|^2dx+\sum_{\al,\beta\in \mathcal{I}} b_{\al\beta}\int n_\al n_{\beta}dx.
\end{align}
If we integrate \eqref{S_n_al_time_evolution}, the quantity $\sum_\al|\na \sqrt{n_\al}|^2_{L_t^2(0,T;L^2_x)}$ will appear on the right hand side. Therefore, we need to estimate the other terms in \eqref{S_n_al_time_evolution}. Before going into the detailed estimates of the second term on the right hand side of \eqref{S_n_al_time_evolution}, we recall that the total mass in the superlevel set can be estimated in terms of the entropy bound $C_{L\log L}$
\begin{align}\label{eta(K)}
\sum_{\al\in \mathcal{I}} \int_{n_\al\geq K}n_\al dx\leq \frac{1}{\log(K)} \sum_{\al\in\mathcal{I}}\int |n_\al\log n_\al |dx \leq \frac{C_{L\log L}}{\log(K)}=:\eta(K).
\end{align}
If $K$ is chosen large compared to the bound $C_{L\log L}$, the constant $\eta(K)$ will be small. It is classical to use this fact to control the nonlinearity in the PKS equation. Now the second term on the right hand side of \eqref{S_n_al_time_evolution} can be estimated using H\"older's inequality, Gagliardo-Nirenberg-Sobolev inequality and Young's inequality as follows:
\begin{align}
\sum_{\al,\beta} b_{\al\beta}&\int n_\al n_{\beta}dx\nonumber\\
 \leq & \max_{\al,\beta} |b_{\al\beta}|\sum_\al |n_\al|_2 \sum_\beta| n_\beta|_2 \nonumber\\
\leq&\max_{\al,\beta} |b_{\al\beta}|\left(\sum_\al |n_\al \mathbf{1}_{n_\al\geq K}|_2+\sum_\al M_\al^{1/2} K^{1/2}\right)^{2} \label{S_n_al_time_evolution_1}\\
\leq&2\max_{\al,\beta}|b_{\al\beta}|\left(\sum_\al |n_\al \mathbf{1}_{n_\al\geq K}|_1^{1/4}|n_{\al}|_3^{3/4}\right)^2+2\max_{\al,\beta}|b_{\al\beta}|\mathcal{I}|K\sum_\al M_\al \nonumber\\
\leq&\eta(K)^{1/2}C_{GNS}\max_{\al,\beta}|b_{\al\beta}|\bigg(\sum_\al M_\al^{1/2}\bigg)\left(\sum_\al|\na\sqrt{n_{\al}}|_2^{2}\right)+2\max_{\al,\beta}|b_{\al\beta}|\mathcal{I}|K\sum_\al M_\al.\nonumber
\end{align}
Combining \eqref{S_n_al_time_evolution} and \eqref{S_n_al_time_evolution_1}, we have the following estimate on the time evolution of $\displaystyle{\sum_\al S[n_\al]}$:
\begin{align*}
\frac{d}{dt}&\sum_{\al} S[n_\al]\\
 & \leq -\sum_\al\left(4-\eta(K)^{1/2}C_{GNS}\max_{\al,\beta}|b_{\al\beta}|\bigg(\sum_\al M_\al^{1/2}\bigg)\right) |\na \sqrt{n_\al}|_2^2 + 2\max_{\al,\beta} |b_{\al\beta}|\cdot |\mathcal{I}|K\sum_{\al}M_\al.
\end{align*}
The coefficient $-(4-\eta(K)^{1/2}C_{GNS}\max_{\al,\beta}|b_{\al\beta}|(\sum_\al M_\al^{1/2}))$ is negative for $K$ large enough. Therefore, for large enough $K$, we have the following estimate:
\begin{align}
\sum_\al\int_0^T\int|\na\sqrt{n_\al}|^2 dxdt\leq& \frac{\displaystyle S[\mathbf{n}(0)]-S[\mathbf{n}(T)]+2\max_{\al,\beta} |b_{\al\beta}|\cdot |\mathcal{I}|K\sum_\al M_\al  T}{\displaystyle 4-\eta(K)^{1/2}C_{GNS}\max_{\al,\beta}|b_{\al\beta}|(\sum_\al M_\al^{1/2})}<\infty.
\end{align}
Since the entropy $S[\mathbf{n}(T)]$ is bounded, the right hand side is bounded. This completes the proof of \eqref{A_priori_est_of_n_c_al_ep_3}. 

Finally, we prove the estimate \eqref{A_priori_est_of_n_c_al_ep_4}. The term $|\sqrt{n_\al^\ep}\na c_\al^\ep|_2^2$ naturally arises when we calculate the time evolution of $\displaystyle{\sum_\al\int n_\al^\ep c_\al^\ep dx}$
\begin{align*}
\frac{1}{2}\frac{d}{dt}\sum_\al \int n_\al^\ep c_{\al}^\ep dx
=&\sum_\al \int n_\al^\ep \de c_{\al}^\ep +\sum_{\al}\int n_\al^\ep |\na c_{\al}^\ep|^2dx.
\end{align*}
Integration in time yields that
\begin{align}
\sum_\al\int_0^T\int n^\ep_\al |\na c_\al^\ep|^2dxdt=\frac{1}{2}\int n_\al^\ep(T)c_\al^\ep(T)-\frac{1}{2}\int n_\al^\ep (0)c_\al^{\ep}(0)dx-\sum_\al\int_0^T\int n_\al^\ep \de c_\al^\ep dxdt\label{A_priori_4_1}.
\end{align}
We first estimate the first term on the right hand side of \eqref{A_priori_4_1}. Applying the estimate of $|n_\al^\ep\log n_\al^\ep|_{L_t^\infty(0,T;L_x^1)}$ \eqref{A_priori_est_of_n_c_al_ep_2}, the relation  $c_\al^\ep =\sum_\beta b_{ \al\beta} K^\ep* n_\beta^\ep $ and the Young's inequality $\displaystyle{ab\leq e^{a-1}+b\ln b,\enskip \forall a,b\geq 1}$, we deduce that
\begin{align*}
|c^\ep_\al(x)|\leq&\frac{1}{2\pi}\sum_{\beta\in \mathcal{I}}|b_{\al\beta}| \int_{|x-y|\leq 1} |K^\ep(|x-y|) n_\beta(y)|dy+\frac{1}{2\pi}\sum_{\beta\in \mathcal{I}}|b_{\al\beta}|\int_{|x-y|\geq 1}|K^\ep (|x-y|)n_\beta(y)|dy\\
\leq & \sum_{\beta\in \mathcal{I}}|b_{\al\beta}|\int_{|x-y|\leq 1}\left((1+n_\beta(y))\log (1+n_\beta(y))+\frac{1}{|x-y|}\right)dy\\
& +\sum_{\beta\in \mathcal{I}}|b_{\al\beta}|\int (\log(1+|x|)+\log(1+|y|))n_\beta(y)dy\\
\lesssim &\sum_{\beta\in \mathcal{I}}|b_{\al\beta}|(C_{L\log L}+M_\beta+1+V_\beta+M_\beta \log( 1+|x|)).
\end{align*}
Combining it with the second moment control \eqref{A_priori_est_of_n_c_al_ep_1}, we have that $\int n_\al c_\al(t)$ is bounded independent of $\ep$ on time interval $[0,T]$:
\begin{align}
\int n_\al  c_{\al} dx\lesssim\sum_{\beta\in \mathcal{I}}|b_{\al\beta}|(C_{L\log L}+M_\beta+1+V_\beta)M_\al+\sum_{\beta\in\mathcal{I}}|b_{\al\beta}|M_\beta V_\al<\infty.\label{A_priori_4_2}
\end{align}
The last term on the right hand side of \eqref{A_priori_4_1} can be estimated using the $L^2([0,T]\times\rr^2)$ estimate of $\na \sqrt{n_\al ^\ep}$ \eqref{A_priori_est_of_n_c_al_ep_3} and the relation
\begin{align*}
\frac{d}{dt}\sum_\al S[n_\al^\ep(t)]=-4\sum_\al\int |\na\sqrt{n_\al^\ep}|^2dx+\sum_{\al\in \mathcal{I}}\int n_\al ^\ep (-\de c_{\al}^\ep)dx.
\end{align*}
Time integration of this relation yields that
\begin{align*}
\bigg|\sum_{\al\in \mathcal{I}}\int_0^T\int n_\al^\ep (-\de c_{\al}^\ep)dxdt\bigg|&=\bigg|\sum_\al S[{n_\al}^\ep(T)]-\sum_\al S[{n_\al}^\ep(0)]+4\sum_\al\int_0^T\int |\na\sqrt{n_\al^\ep}|^2dxdt\bigg|\\ 
& \leq C(C_{L\log L})<\infty.
\end{align*}
Combining this estimate, \eqref{A_priori_4_1} and \eqref{A_priori_4_2}, we completed the proof of \eqref{A_priori_est_of_n_c_al_ep_4}.
In this way, we obtained estimates on the two terms appearing in the dissipation of the free energy.

\medskip\noindent
\textbf{$\bullet$ STEP \#2}. Passing to the limit in $L_t^2(\delta,T; L^2)$ for any $\delta>0$. Here we would like to use the Aubin-Lions compactness lemma:
\begin{lem}[Aubin-Lions lemma, \cite{BlanchetCarrilloMasmoudi08}]
Take $T> 0$ and $1<p<\infty$. Assume that $(f_n)_{n\in \mathbb{N}}$ is a bounded sequence of functions in $L^p([0,T]; H)$ where $H$ is a Banach space. If $(f_n)_{n\in\mathbb{N}}$ is also bounded in $L^p([0,T];V)$ where $V$ is compactly embedded in $H$ and $(\pa f_n/\pa_t)_{n\in\mathbb{N}} \subset L^p([0,T];W)$ uniformly with respect to $n\in \mathbb{N}$ where $H$ is imbedded in $W$, then $(f_n)_{n\in \mathbb{N}}$ is relatively compact in $L^p([0, T]; H)$.
\end{lem}
Our goal is to find the appropriate spaces $V,H,W$ for $(n_\al^\ep)_{\ep>0}$. We subdivide the proof into steps, each step determines one space in the lemma. We will  show that the following estimates are satisfied by the regularized solutions with the constant $C_{L_t^2H_x^1}$ independent of the regularized parameter $\ep$:
\begin{align*}
&|n_\al ^\ep|_{ L_t^2([\delta,T], L_x^2)}\leq C_{L_t^2H_x^1}<\infty,\\
&|\na n_\al^\ep|_{ L_t^2([\delta,T], L_x^2)}\leq C_{L_t^2H_x^1}<\infty, \quad \forall\al \in \mathcal{I}.
\end{align*}

We begin with the $H\!=\!L^2$- estimate of $\sum_\al|n_\al^\ep|_{L^2_t([\delta, T];L_x^2)}^2$.
Here we prove that the solutions ${n}_\al^\ep(t),\enskip \forall \al \in\mc I$ are $L^2$ integrable in space for $\forall t\in[\delta,T]$. If the initial data $n_{\al 0}$ is $L^2$ integrable for all $\al$, the solutions to the regularized equation \eqref{EQ:KS_multiple_groups_regularized} stay in $L^2$ for all time. This is the content of Lemma \ref{L_p_bound_n_al}. However, the initial constraint \eqref{General_condition_for_initial_data} does not guarantee $L^p$ boundedness, so we prove the hypercontractivity property of the equation \eqref{EQ:KS_multiple_groups}, which yields that the solutions become $L^2$ integrable after an arbitrarily small amount of time $\delta>0$. This is the content of Lemma \ref{lem:Hypercontractivity_estimate}.
\begin{lem}\label{L_p_bound_n_al}
Consider the regularized multi-species PKS system \eqref{EQ:KS_multiple_groups_regularized} subject to initial condition $n_{\al 0}\in L^p$,  $\forall \al\in \mathcal{I},$ $\forall p\in [1, \infty)$. If the assumptions in the Proposition \ref{pro 2.1 multiple group} are satisfied, then the solutions to the system \eqref{EQ:KS_multiple_groups_regularized} are bounded in $L^p$ for $\forall t\in[0,T]$.
\end{lem}
\begin{proof}
The $p=1$ case is equivalent to the fact that the regularized equations preserve mass.

We do the $L^p$ energy estimate formally, i.e., we assume $-\de c_\al=\sum_\beta b_{\al\beta} n_\beta$,  and refer the interested readers to the paper  \cite{BlanchetEJDE06} for detailed justifications. During the calculation, we will use the following natural implication of the GNS inequality
\begin{align}
\int (f-K)_+^{p+1}dx&\leq C_{GNS} \int (f-K)_+dx\int |\na (f-K)_+^{p/2}|^2dx\nonumber \\
 & \leq  C_{GNS} \frac{|f\log f|_1}{\log K} \int |\na (f-K)_+^{p/2}|^2dx =:C_{GNS} \eta(K) \int |\na (f-K)_+^{p/2}|^2dx.\label{f-K_+_GNS}
\end{align}
Note that if $|f\log f|_1$ is bounded, 
$\eta (K)$ is small if one choose $K$ large. Now we estimate the time evolution of $\sum_\al |(n_\al-K)_+|_p^p$ with \eqref{f-K_+_GNS} as follows
\begin{align*}
\frac{1}{p}\sum_{\al}&\frac{d}{dt}\int (n_\al -K)_+^pdx\\
=&-4\frac{p-1}{p^2}\sum_\al\int |\na (n_\al -K)_+^{p/2}|^2dx-\sum_{\al}\frac{1}{p}\int \na c_{\al}\cdot\na (n_\al -K)_+^pdx\\
 & -\sum_{\al} \int\de c_{\al}  n_\al (n_\al -K)_+^{p-1}dx\\
\leq&-4\frac{p-1}{p^2}\sum_\al\int |\na (n_\al -K)_+^{p/2}|^2dx+\frac{p+1}{p}\sum_{\al,\beta}|b_{\al\beta}|\int (n_\al -K)_+^p(n_{\beta}-K)_+dx\\
&+\frac{p+1}{p}K\sum_{\al,\beta}|b_{\al\beta}|\int (n_\al-K)_+^pdx+K\sum_{\al,\beta}|b_{\al\beta}|\int(n_\beta-K)_+(n_\al-K)_+^{p-1}dx\\&+K^2\sum_{\al,\beta}\int |b_{\al\beta}|(n_\al-K)_+^{p-1}dx
\end{align*}
and hence we find
\begin{align*}
\frac{1}{p}\sum_{\al}&\frac{d}{dt}\int (n_\al -K)_+^pdx\\
\leq&-4\frac{p-1}{p^2}\sum_\al\int |\na (n_\al-K)_+^{p/2}|^2dx\\
 &+\max_{\al}\left(\sum_\beta |b_{\al\beta}|\right)C_{GNS}\sum_\al |(n_\al-K)_+|_1|\na(n_\al-K)_+^{p/2}|_2^{2}\\
&+C_p(K,\mathbf{B},\mathbf{M})|(n_\al-K)_+|_p^p+C_p(K,\mathbf{B},\mathbf{M})\\
\leq&\left(-\frac{4(p-1)}{p^2}+\eta(K)\max_\al\left( \sum_\beta|b_{\al\beta}|\right)C_{GNS}\right)\sum_\al \int |\na (n_\al-K)_+^{p/2}|^2dx\\
&+C_p(K,\mathbf{B},\mathbf{M})\sum_\al|(n_\al-K)_+|_p^p+C_p(K,\mathbf{B},\mathbf{M}).
\end{align*}
Due to the estimates \eqref{A_priori_est_of_n_c_al_ep_2} and \eqref{eta(K)}, the constant $\eta(K)$ can be made small enough such that the leading order term is negative, and the estimate can be further simplified as follows:
\begin{equation}\label{(n-K)_+_Lp_time_evolution}
\frac{d}{dt}\sum_{\al}|(n_\al -K)_+|_p^p\leq C_p(K,\mathbf{B},\mathbf{M})\sum_\al|(n_\al-K)_+|_p^p+C_p(K,\mathbf{B},\mathbf{M}).
\end{equation}
Now we see that for any finite time interval $[0, T]$, the $L^p$ norm is bounded uniformly independent of $\ep$.
\end{proof}

\begin{lem}\label{lem:Hypercontractivity_estimate}
Consider the regularized multi-species PKS system \eqref{EQ:KS_multiple_groups_regularized} subject to initial data $\mathbf{n}_0$ satisfying \eqref{General_condition_for_initial_data}. If the assumptions in Proposition \ref{pro 2.1 multiple group} is satisfied, then there exists a continuous function $h_p\in C(\rr_+)$ such that for almost any $t>0$, $|n(\cdot ,t)|_p\leq h_p(t)$.
\end{lem}
\begin{proof}
The proof is similar to the corresponding proof in \cite{BlanchetEJDE06} with some modifications. For the sake of completeness, we sketch the proof. First, we fix $t>0$ and $1<p<\infty$, and define \begin{equation}
q(s):=1+(p-1)\frac{s}{t},\quad q\in[1,p] \text{ for }s\in[0,t].
\end{equation}
Next, we define the following quantities:
\begin{align}
\mathbb{F}_\al(s)=&\bigg(\int_{\rr^2}(n_\al(x,s)-K)_+^{q(s)}dx\bigg)^{1/q(s)},\\
\mathbb{F}(s)=&\bigg(\sum_\al \mathbb{F}_\al^{q(s)}(s)\bigg)^{1/q(s)}.
\end{align}
By taking the $s$ derivative of the function $\mathbb{F}^{q(s)}(s)$, we obtain the following relation
\begin{align*}
\frac{d}{ds}\sum_\al\int(n_\al(x,s)-K)_+^{q(s)}dx=q(s)\mathbb{F}^{q(s)-1}\frac{d}{ds}\mathbb{F}+\frac{dq(s)/ds}{q(s)}\mathbb{F}^{q(s)}\log \mathbb{F}^{q(s)}.
\end{align*}
Combining it with the log-Sobolev inequality
\begin{align*}
\int f^2\log\left(\frac{f^2}{\int f^2 dx}\right)dx\leq 2\sigma\int|\na f|^2dx-(2+\log(2\pi \sigma))\int f^2 dx,\quad \forall \sigma>0,
\end{align*}
and the same argument to prove \eqref{(n-K)_+_Lp_time_evolution}, we end up with the following estimate, inside which the notation $(\cdot)'$ is used to represent $\frac{d}{ds}$,
\begin{align}
\mathbb{F}^{q-1}\frac{d}{dt}\mathbb{F}=&\frac{q'}{q^2}\sum_\al \int (n_\al-K)_+^{q}\log \frac{(n_\al-K)_+^q}{\mathbb{F}^q}dx+\sum_\al\int(n_\al-K)_+^{q-1}\pa_s n_\al dx\nonumber\\
\leq&\frac{q'}{q^2}\sum_\al \int (n_\al-K)_+^{q}\log \frac{(n_\al-K)_+^q}{\mathbb{F}_\al^q}dx+\sum_\al\int(n_\al-K)_+^{q-1}\pa_s n_\al dx\nonumber\\
\leq&\sum_\al\bigg(\frac{2\sigma q'}{q^2}-4\frac{q-1}{q^2}+C(\mathbf{B})\eta(K)\bigg)|\na(n_\al-K)_+^{q/2}|_2^2\nonumber\\
&+\sum_\al \left((-2-\log(2\pi\sigma))\frac{q'}{q^2}+C(q,\mathbf{B}, \mathbf{M},K)\right)\int(n_\al-K)^q_+dx+C(q,\mathbf{B},\mathbf{M},K).
\end{align}
Here the constants $C(q,\mathbf{B},\mathbf{M},K)$ depends on the parameter $q$. However, since $q$ is lying in a compact set $[0,p]$ on the time interval $[0,t]$, it can be chosen such that it only depends on the fixed parameter $p$. Now by taking $\sigma$ small enough, we end up with the following differential inequality
\begin{align*}
\mathbb{F}^{q-1}\frac{d}{ds}\mathbb{F}\leq \left((-2-\log(2\pi\sigma))\frac{q'}{q^2}+C(p,\mathbf{B}, \mathbf{M},K)\right)\mathbb{F}^q+C(p,\mathbf{B},\mathbf{M},K).
\end{align*}
Combining the fact that $\mathbb{F}(0)$ is finite and the coefficient $(-2-\log(2\pi\sigma))\frac{q'}{q^2}+C(p,\mathbf{B}, \mathbf{M},K)$ is time integrable on $[0,t]$  and applying standard ODE estimates, we obtain that $\mathbb{F}\leq h_p(t)$. This finishes the proof of the lemma. 
\end{proof}

We now turn to the $V$-space estimates, where $V:=H^1\cap\{f|\int f|x|^2 dx<\infty\}$: $\sum_\al|\na n_\al^\ep |^2_{L_t^2([\delta, T];L^2_x)}$.
In order to get the $L_t^2([\delta,T]; L_x^2)$ control of the $\na n_\al^\ep$, we first calculate the time evolution of $\sum|n_\al ^\ep|_2^2$:
\be\ba
\frac{d}{dt}\sum_\al\int |n_\al ^\ep|^2dx
=&-2\sum_\al\int |\na n_\al ^\ep|^2dx+2\sum_{\al}\int \na n_\al ^\ep \cdot \na c_{\al}^\ep n_\al^\ep dx.
\ea\ee
Integration in time yields that
\begin{align}
\sum_\al|n_\al ^\ep(T)|_2^2-\sum_\al|n_\al ^\ep(\delta)|_2^2
+\sum_\al| \na n_\al ^\ep|^2_{L^2_t([\delta,T];L^2_x)}\leq&
\sum_\al| n_\al ^\ep\na c_{\al}^\ep|_{L_t^2([\delta,T];L_x^2)}^2.\label{n_ep_al_L2}
\end{align}
We see that since $|n_\al^\ep|_{L^{\infty}_t(\delta,T;L_x^2)}$ is bounded independent of $\ep$, if the right hand side $\sum_\al|n^\ep_\al \na c^\ep_\al|_{L_t^2([\delta,T];L_x^2)}$ is bounded, $|\na n_\al^\ep|_{L^2_t(\delta,T;L_x^2)}$ will be bounded independent of $\ep$. By the HLS inequality, we have that
\be
|\na c_\al ^\ep|_4\leq C_{HLS}\sum_{\beta\in \mathcal{I}}|b_{\al\beta}|\cdot|n_\beta ^\ep |_{4/3}.
\ee
As a result, we have that
\be
|n_\al ^\ep \na c_{\al}^\ep |_2\leq |n_\al ^\ep |_{4}|\na c_{\al}^\ep |_4\leq \sum_{\beta} C_{HLS}|b_{\al\beta}|\cdot|n_\al ^\ep|_{4}|n_{\beta}^\ep|_{4/3}.
\ee
Since $n_\al^\ep$ is bounded independent of $\ep$ in the space $ L_t^\infty(\delta,T;L_x^p)$, $\forall \al\in \mathcal I$, $\forall  p\in(1,\infty)$, the product $n^\ep\na c^\ep $ is bounded in $L_t^\infty(\delta,T;L_x^2)$. Combining this fact and the estimate \eqref{n_ep_al_L2}, we have that  $\sum_\al|\na n_\al^\ep|_{L_t^2(\delta,T;L_x^2)}^2$ is bounded independent of $\ep$.

Define the space $V$ as $H^1\cap\{f|\int f|x|^2 dx<\infty\}$. A bounded set in the space $V$  is precompact in $L^2$. Combining the second moment bound \eqref{Second_moment_control} and the $H^1$ bound of $(n_\al^\ep)_{\al\in\mathcal{I}}$, we have that the set $(n_\al^\ep)_{\ep\geq 0}$, $\forall \al \in \mathcal{I}$ lies in a compact subspace of $L^2$ for almost every $t\in [\delta, T]$.
Finally, the $W$-estimate where  $W:=H^{-1}$: $\sum_\al|\pa_t n_\al^\ep|^2_{L_t^2(\delta,T; H_x^{-1})}$  is relatively straightforward thanks to the equation \eqref{EQ:KS_multiple_groups}.

\medskip\noindent
\textbf{$\bullet$ STEP \#3}. Proof of the free energy dissipation inequality \eqref{free_energy_dissipation}.
Since the solution to the regularized multi-species PKS system has a decreasing free energy $E[\mathbf{n}^\ep]$, we have that
\begin{equation}
E[\mathbf{n}^\ep(\delta)]\geq E[\mathbf{n}^\ep(t)]+\sum_\al\int_\delta^t\int n_\al^\ep |\na \log n_\al^\ep - \na c^\ep_{\al}|^2dxdt,\quad \forall t\in[\delta, T].\label{free_energy_dissipation_0}
\end{equation}
In order to show \eqref{free_energy_dissipation}, we need to show proper convergence for each single term in \eqref{free_energy_dissipation_0}. We first decompose the free energy dissipation term as follows:
\begin{align}
\sum_\al\int_\delta^T\int_{\rr^2} n_\al^\ep |\na \log n_\al^\ep &-\na c^\ep_{\al}|^2dxdt\nonumber\\
=& 4\sum_\al\int_\delta^T\int_{ \rr^2}|\na\sqrt{n_\al^\ep }|^2dxdt+\sum_\al\int_\delta^T\int_{ \rr^2}n_\al^\ep |\na c^\ep_{\al}|^2dxdt \label{free_energy_dissipation_1}\\
 &-2 \sum_{\al,\beta}\int_\delta^T\int_{ \rr^2} b_{\al\beta} n_\al^\ep  n_{\beta}^\ep dxdt.\nonumber
\end{align}
By the convexity of $f\rightarrow \int_{\rr^2}|\na \sqrt{f}|^2dx$, weak semi-continuity and the strong convergence of $n_\al^\ep$ in $L^2_t([\delta,T];L^2_x)$, we have that the first two terms in \eqref{free_energy_dissipation_1} satisfies the following inequalities
\begin{align}
\int_\delta^T\int_{ \rr^2}|\na\sqrt{n_\al }|^2dxdt\leq&\liminf_{\ep\ra 0_+}\int_\delta^T\int_{\rr^2}|\na\sqrt{n_\al ^{\ep}}|^2dxdt\label{na_sqrt_n_L2L2}\\
\int_\delta^T\int_{ \rr^2}n_\al |\na c_{\al}|^2dxdt=&\lim_{\ep\ra 0_+}\int_\delta^T\int_{\rr^2}n_\al ^{\ep}|\na c_{\al}^{\ep}|^2dxdt.\label{sqrt_n_na_c_L2L2}
\end{align}
Since the $(n_\al^\ep)_{\ep> 0}$ converges strongly in the $L^2([\delta, T]\times \rr^2)$ space. The last term on the right hand side of \eqref{free_energy_dissipation_1} converges.  Moreover, it can be checked that $S[{n}_\al^\ep(t)]\rightarrow S[{n}_\al(t)]$ for almost every $t\in[\delta,T]$. 
The argument is similar to the one used in \cite{BlanchetEJDE06} Lemma 4.6. As a result, combining these facts and \eqref{free_energy_dissipation_0}, \eqref{free_energy_dissipation_1}, \eqref{na_sqrt_n_L2L2} and \eqref{sqrt_n_na_c_L2L2} yields that
\begin{align*}
E[\mathbf{n}(\delta)]\geq E[\mathbf{n}(t)]+\sum_\al\int_\delta^t\int n_\al|\na \log n_\al-\na c_\al|^2dxds.
\end{align*}
Now by the monotone convergence theorem and a Cantor  diagonal argument, we have proven \eqref{free_energy_dissipation}.
\end{proof}

\begin{proof}[Proof of proposition \ref{pro 1.2 multiple group}]
We prove by contradiction. Assume that at time $T_\star<\infty$, the entropy $\sum_\al S[n_\al^\ep(T_\star)]$ is uniformly bounded with respect to $\ep$.

First, from the equation \eqref{EQ:KS_multiple_groups_regularized}, we directly calculate the time evolution of the entropy:
\begin{align}
\frac{d}{dt}\sum_\al\int n_\al^\ep\log n_\al^\ep dx=&-\sum_\al 4\int|\na\sqrt{n^\ep}|^2dx-\sum_{\al,\beta} b_{\al\beta}\int_{n^\ep_\al\leq K}n_\al^\ep \de(K^\ep*n^\ep_\beta)dx\nonumber\\
& -\sum_{\al,\beta} b_{\al\beta}\int_{n^\ep_\al>K}n_\al^\ep\de(K^\ep *n^\ep_\beta)dx\label{extension_theorem_proof_1}\\
=:&-\sum_\al 4\int|\na\sqrt{n_\al^\ep}|^2dx+I+II.\nonumber
\end{align}
The term $I$ in \eqref{extension_theorem_proof_1} can be estimated as follows:
\begin{align}
I\leq & \sum_{\al,\beta}K|b_{\al\beta}|\de K^\ep|_1M_\beta.\label{extension_theorem_proof_2}
\end{align}
Recall that $|\de K^\ep|_1$ is bounded independent of $\ep$, so term $I$ is bounded independent of $\ep$. For the term $II$ in \eqref{extension_theorem_proof_1}, we estimate it using the H\"older's inequality, Gagliardo-Nirenberg-Sobolev inequality and Young's inequality as follows:
\begin{align}
II\leq&\sum_{\al,\beta}|b_{\al\beta}|\bigg(\int_{n_\al^\ep\geq K}(n_{\al}^\ep)^2dx+|\de K^\ep|_1^2*|n_{\beta}^\ep|_2^2\bigg)\nonumber\\
\leq&\sum_{\al,\beta}|b_{\al\beta}|\left(\left(\int_{n_\al^\ep\geq K}n^{\ep}_\al dx\right)^{1/2}|n_\al^\ep|_3^{3/2}+|\de K^\ep|_1^2\left(M_\beta K+\int_{n_\beta^\ep\geq K}(n_\beta^\ep)^2dx\right) \right)\nonumber\\
\leq&\sum_{\al,\beta}|b_{\al\beta}|\left(\frac{{S}^{1/2}_+[n_\al]}{(\log K)^{1/2}}C_{GNS}|n_\al^\ep|_1^{1/2}|\na \sqrt{n^\ep_\al}|_2^2\right.\label{extension_theorem_proof_3}\\
& \left.\qquad \qquad \ +C_{GNS}|\de K^\ep|_1^2\frac{{S}^{1/2}_+[n_\beta]}{(\log K)^{1/2}}M_\beta^{1/2}|\na \sqrt{ n^\ep_\beta}|_2^2+|\de K^\ep|_1^2M_\beta K\right)\nonumber\\
\leq&\sum_{\al,\beta}|b_{\al\beta}|C_{GNS}(1+|\de K^\ep|_1^2)\frac{{S}^{1/2}_+[n_\al]}{(\log K)^{1/2}}M_\al^{1/2}|\na \sqrt{n_\al^\ep}|_2^2+\sum_{\al,\beta}|b_{\al\beta}|\cdot|\de K^\ep|_1^2 M_\al K.\nonumber  
\end{align}
Here $S_+$ denote the positive part of the entropy, i.e., $S_+[f]=\int f\log^+fdx$. Combining the estimates \eqref{extension_theorem_proof_1}, \eqref{extension_theorem_proof_2} with \eqref{extension_theorem_proof_3}, we end up with
\begin{align*}
\frac{d}{dt}\sum_\al S [n_\al^\ep]  \leq& \sum_\al\underbrace{\left(-4+\sum_{\beta}|b_{\al\beta}|C_{GNS}(1+|\de K^\ep|_1^2)\frac{{S}^{1/2}_+[n_\al^\ep]}{(\log K)^{1/2}}M_\al^{1/2}\right)}_{=:A(t)}|\na \sqrt{n_\al^\ep}|_2^2\\
 & +\sum_{\al,\beta}|b_{\al\beta}|(1+|\de K^\ep|_1^2) M_\al K.
\end{align*}
Since the negative part of the entropy and the second moment are bounded \eqref{negative_part_of_entropy}, \eqref{Second_moment_control}, we have that $A(t)$ can be estimated as follows:
\begin{align}
A(t)\leq& -4+\frac{C_{GNS}}{(\log K)^{1/2}}\sum_\beta|b_{\al\beta}|(1+|\de K^\ep|_1^2)M_\al^{1/2}\bigg( S[n_\al^\ep(t)]+\frac{1}{2}V(T_\star)\nonumber\\
&+ \frac{1}{2}\left(4\sum_\al M_\al+\sum_{\al,\beta}\frac{(b_{\al\beta})_-M_\al M_\beta}{2\pi}\right)(t-T_\star)+\log (2\pi)M_\al+e^{-1} \bigg)^{1/2}
\end{align}
Since the entropy $\sum_\al S[n_\al^\ep]$ is uniformly bounded independent of $\ep$ at time $T_\star$, we could take the $K$ large such that $A(t)\leq -2$ at time $T_\star$. By continuity, there is a small time $\tau_\ep$ such that for $\forall t\in[T_\star, T_\star+\tau_\ep)$,
\begin{equation}
\sum_{\al} S[n_\al^\ep(t)]\leq \sum_{\al}S[n_\al^\ep(T_\star)]+(t-T_\star)\sum_{\al,\beta}|b_{\al\beta}|(1+|\de K^\ep|_1^2) M_\al K,\quad \forall t\in [T_\star, T_\star+\tau_\ep].
\end{equation}
But then we can pick $\tau$ independent of $\ep$ such that
\begin{align*}
A(t)\leq& -4+\frac{C(\mathbf{B},\mathbf{M})}{(\log K)^{1/2}}\bigg(\sum_\al S[n_\al^\ep(T_\star)]+K\tau+1\bigg)\leq 0.
\end{align*}
The solution $\tau$ to the above inequality is independent of the choice of $\ep$, and $[T_\star, T_\star+\tau)\subset [T_\star, T_\star+\tau_\ep)$ for any $\ep$. Therefore, by Proposition \ref{pro 2.1 multiple group}, we can extend the free energy solution pass the $T_\star$, contradicting the maximality of $T_\star$. As a result, we have completed the proof of the proposition.
\end{proof}
\section{Smoothness of the free energy solutions}
In this section, we prove Theorem \ref{Theorem_of_smoothness_of_solutions}. The proof is similar to the arguments in \cite{EganaMischler16}. For the sake of brevity, we skip some details and emphasize the main differences. The proof is decomposed into several lemmas. We first introduce the concept of Fisher information and renormalized solutions, then prove the $L^p$ integrability of the physically relevant free energy solutions and use standard parabolic equation technique to improve it to $C^\infty$ regularity, and conclude with the proof of the free energy equality.

First note from the physical restrictions \eqref{Total_second_moment_evolution} and \eqref{A_t_Bound} that we have bounded entropy and free energy dissipation, i.e., $\mathcal{A}_t[\mathbf{n}]<\infty$ and bounded second moment $ V[\mathbf{n}(t)]$ for all $t\in [0,T_\star)$, where $T_\star$ is the maximal existence time.

Next we present the following time integral bound for the Fisher information
\begin{lem}\label{Lem:I_time_integral}
If the conditions in the Theorem \ref{Theorem_of_smoothness_of_solutions} are satisfied, for any physically relevant free energy solutions to \eqref{EQ:KS_multiple_groups} and any time $T\in [0,T_\star)$, there exists a constant $C_F$ such that the Fisher information of the solution
\begin{equation}
F[n_\al]:=\int_{\rr^2}\frac{|\na n_\al|^2}{n_\al}dx,
\end{equation}
is time integrable, i.e.,
\begin{equation}
\sum_{\al\in \mathcal{I}}\int_0^T F[n_\al(t)]dt\leq C_F\left(M, T,\mathcal{A}_T[\mathbf{n}], \sup_{t\in [0,T)}\sum_\al V_\al(t)\right),\quad T\in[0,T_\star).
\end{equation}
\end{lem}
\begin{proof}
The proof is essentially the same as the corresponding one in the single species case. For the sake of brevity, we skip the proof here and refer the interested readers to the proof of Lemma 2.2 and the remark after in the paper \cite{EganaMischler16} for further details.
\end{proof}
\begin{rmk}
For the supercritical mass case, one can use the relative entropy method to derive the boundness of the entropy and entropy dissipation $\mathcal{A}_T[\mathbf{n}]$ before the blow-up time $T_\star$. We refer the interested reader to the papers \cite{BlanchetCarrilloMasmoudi08} and \cite{EganaMischler16} for further details. 
\end{rmk}

The next lemma enable us to take advantage of choosing different renormalizing functions in the later proof.
\begin{lem}\label{Lem:Renormalized relation}
Any physically relevant free energy solutions $\mathbf{n}$ to \eqref{EQ:KS_multiple_groups} satisfy the following estimate for any times $0\leq t_0\leq t_1<T_\star$
\begin{align}
\int_{\rr^2}\Gamma(n_{\al}(x,t_1))&dx+\int_{t_0}^{t_1} \int_{\rr^2} \Gamma''(n_\al(x,s))|\na n_\al(x,s)|^2dxds\nonumber\\
\leq &\int_{\rr^2} \Gamma(n_\al(x,t_0))dx \nonumber\\
& \ +\int_{t_0}^{t_1}\int_{\rr^2} \bigg((\Gamma'(n_\al(x,s))n_\al(x,s)-\Gamma(n_\al(x,s)))\bigg(\sum_{\beta\in \mathcal{I}}b_{\al\beta}n_\beta(s)\bigg)\bigg)_+dxds\label{Renormalized relation}\\
\leq &\int_{\rr^2} \Gamma(n_\al(x,t_0))dx\nonumber\\
& \ +\sum_{\beta\in \mathcal{I}}|b_{\al\beta}|\int_{t_0}^{t_1}\int_{\rr^2} {\left|\Gamma'(n_\al(x,s))n_\al(x,s)-\Gamma(n_\al(x,s))\right| n_\beta(s)dxds},\nonumber
\end{align}
where $\Gamma:\rr\rightarrow\rr$ is an arbitrary convex piecewise $C^1$ function satisfying the following estimates with some constant $C_\beta$
\begin{equation}\label{renormalizing_function_constraint}
|\Gamma(u)|\leq C_\Gamma(1+u(\log u)_+),\quad |\Gamma'(u)u-\Gamma(u)|\leq C_{\Gamma}(1+|u|),\quad \forall u\in \rr.
\end{equation}

\end{lem}
\begin{rmk}
Here in order to analyse the PKS equation \eqref{EQ:KS_multiple_groups} with general chemical generation coefficients, we introduce a stronger restriction on the growth of the normalizing function $\Gamma$ comparing to the paper \cite{EganaMischler16}. Here we assume that the absolute value of the expression $\Gamma'(u)u-\Gamma(u)$ grows at most linearly at infinity, whereas in the paper \cite{EganaMischler16}, it is only assumed that the positive part $(\Gamma'(u)u-\Gamma(u))_+$ grows at most linearly.
\end{rmk}
\begin{proof}
The proof is essentially the same as the proof of Lemma 2.5 in the paper \cite{EganaMischler16}. For the sake of simplicity, we do a formal computation and refer the interested readers to \cite{EganaMischler16} for further justifications. By applying the chain rule, we obtain
\begin{equation}
\pa_t\Gamma(n_\al)=\de \Gamma(n_\al)-\Gamma''(n_\al)|\na n_\al|^2-\na c_\al\cdot \na \Gamma(n_\al)-\Gamma'(n_\al)\de c_{\al} n_\al,\quad \forall\al\in\mathcal{I}.
\end{equation}
Now test it against an arbitrary smooth function $\chi\in \mathcal{D}(\rr^2)$ and use the relation $-\de c_\al=\sum_\beta b_{\al\beta}n_\beta$, we have the following relation:
\begin{align*}
\int_{\rr^2} \Gamma(n_\al(t_1))\chi dx+&\int_{t_0}^{t_1}\int_{\rr^2}\Gamma''(n_\al)|\na n_\al(s)|^2 \chi dxds=\int_{\rr^2} \Gamma(n_\al(t_0))\chi dx\\
+&\int_{t_0}^{t_1}\int_{\rr^2}\left(\Gamma'(n_\al)\sum_{\beta}b_{\al\beta}n_\beta n_\al\chi+\Gamma(n_\al)\de \chi+\Gamma(n_\al)\na \cdot(\na c_\al \chi)\right)dxds.
\end{align*}
Rewrite the above relation using the integration by parts and the fact that $\de c_\al=-\sum_\beta b_{\al\beta}n_\beta$,
\begin{align*}
\int_{\rr^2} &\Gamma(n_\al(t_1))\chi dx+\int_{t_0}^{t_1}\int_{\rr^2}\Gamma''(n_\al)|\na n_\al(s)|^2 \chi dxds\nonumber\\
=&\int_{\rr^2} \Gamma(n_\al(t_0))\chi dx\\
& \ +\int_{t_0}^{t_1}\int_{\rr^2}\left(\left[\Gamma'(n_\al) n_\al-\Gamma(n_\al)\right]\left(\sum_{\beta}b_{\al\beta}n_\beta\right)\chi+\left[\Gamma(n_\al)\de \chi+\Gamma(n_\al)\na c_\al \cdot\na\chi\right]\right)dxds.
\end{align*}
Now take $\chi\rightarrow 1$, we end up with the relation \eqref{Renormalized relation}.

In order to prove the Lemma, one first prove \eqref{Renormalized relation} with renormalizing function $\Gamma_i$, $i\in\mathbb{N}$, which grows at most linearly at infinity. Next one prove the estimate \eqref{Renormalized relation} with renormalizing functions with super linear growth \eqref{renormalizing_function_constraint} by taking limit of the inequalities \eqref{Renormalized relation} subject to approximating linear renormalizing functions $(\Gamma_i)_{i\in\mathbb{N}}$. One use the Lebesgue dominated convergence theorem to guarantee the convergence of the term \begin{align*}
\lim_{i\rightarrow \infty}\int_{t_0}^{t_1}\left([\Gamma_i'(n_\al)n_\al-\Gamma_i(n_\al)]\bigg(\sum_{\beta}b_{\al\beta}n_\beta\bigg)\right)_+dxds.
\end{align*}
However, if the function $\sum_\beta b_{\al\beta}n_{\beta}$ can be either positive or negative, we have to assume that $|\Gamma'(u)u-\Gamma(u)|$ grows at most linearly near infinity.  
\end{proof}

Now we prove the $L^p$ estimate of the solution
\begin{lem}
Consider physically relevant free energy solutions $(n_\al)_{\al\in \mathcal{I}}$ to equation \eqref{EQ:KS_multiple_groups} subject to initial data \eqref{General_condition_for_initial_data}. Let $t_0\in [0,T_\star)$ be the time such that $\sum_{\al\in \mathcal{I}}|n_\al(t_0)|_p<\infty$, for some $p\in [2,\infty)$. Then for all time $t_1\in [t_0,T]\subset[t_0,T_\star)$, there exists a constant $C_p:=C_p(\mathbf{M},T, \sum_{\al\in\mathcal{I}}|n_\al(t_0)|_p,V[\mathbf{n}(t_0)],\mathcal{A}_T)$ such that
\begin{align}\
 \sum_{\al\in \mathcal{I}}|n_\al(t_1)|_p^p+\frac{p-1}{2p}\sum_{\al\in \mathcal{I}}\int_{t_0}^{t_1}|\na (n_\al^{p/2})|_2^2ds\leq C_p,\quad p\in [2,\infty).
\end{align}
\end{lem}

\begin{proof}
The proof is similar to the corresponding one in \cite{EganaMischler16}. We decompose the proof into two steps.

\textbf{Step 1:} We prove a logarithmic improvement to the $L\log L$ integrability. The goal is to show that there exists a constant $C_{S_2}:=C_{S_2}(M,T,\mathcal{A}_T, \sup_{[t_0, T]} V[\mathbf{n}(t)])$ such that the following estimate is satisfied for any $t_1\in [ t_0, T]$,
\begin{align}\label{EQ:H2entropy_estimate}
\sum_\al S_2[n_\al(t_1)]\leq \sum_\al S_2[n_\al(t_0)]+C_{S_2},\quad S_2[f]:=&\int f(\widetilde{\log }f)^2dx,
\end{align}
where the $\widetilde{\log}$ function is the logarithmic function truncated from below:
\begin{align}
\widetilde{\log}u:=\mathbf{1}_{u\leq e}+(\log u)\mathbf{1}_{u>e}.
\end{align}
For the sake of notational simplicity, we further introduce the bounded truncated logarithmic function $\widetilde{\log}_K$ as follows:
\begin{equation}
\widetilde{\log}_K(u):=\mathbf{1}_{u\leq e}+\mathbf{1}_{e<u\leq K}\log u+\mathbf{1}_{u>K}\log K.
\end{equation}

Since $(\cdot)\widetilde\log^2(\cdot)$ does not satisfy the growth constraint \eqref{renormalizing_function_constraint}, we approximate it by the function $\Gamma_K(u)$, $K\geq e^2$, 
\begin{align}\label{betaK}
\Gamma_K(u):=\left\{\begin{array}{rr}u(\widetilde{\log}u)^2,\quad u\leq K;\\
(2+\log K)u\log u -2K\log K, \quad u>K.\end{array}\right.
\end{align}
One can check that the function $\Gamma_K$ is convex and satisfies the properties \eqref{renormalizing_function_constraint} 

\begin{align}
\Gamma_K''(u)\geq 2\frac{\log u}{u}\mathbf{1}_{e\leq u\leq K}+(2+\log K)\frac{1}{u}\mathbf{1}_{u>K}\geq \frac{\widetilde{\log}_K u}{u}\mathbf{1}_{u\geq e}\geq 0,\label{beta_K''_lower_bound}\\
{|\Gamma'_K(u)u-\Gamma_K(u)|\leq 2 u\widetilde{\log}u\mathbf{1}_{u\leq K}+4\log K u\mathbf{1}_{u>K}\leq C_K(1+u)}.\label{beta'u-beta_upper_bound}
\end{align}
Now we estimate the time evolution of $\sum_\al\int \Gamma_K(n_\al)dx$ using the renormalization relation \eqref{Renormalized relation}, the positivity of $b_{\al\beta}$, \eqref{beta_K''_lower_bound}, \eqref{beta'u-beta_upper_bound} and the definition of $\widetilde{\log}, \widetilde{\log}_K$ as follows
\begin{align}
\sum_\al \int \Gamma_K(n_\al(t_1))dx&+\sum_\al\int_{t_0}^{t_1}\int \frac{\widetilde{\log}_K(n_\al)}{n_\al}\mathbf{1}_{n_\al\geq e}|\na(n_\al)|^2dxds\nonumber\\
\leq &\sum_\al \int \Gamma_K(n_\al(t_0))dx\label{H2evolution_1} \\
& +\sum_{\al,\beta}|b_{\al\beta}|\int_{t_0}^{t_1}\int\left(2 n_\al \widetilde{\log } n_\al \mathbf{1}_{n_\al\leq K}+4\log K n_\al \mathbf{1}_{n_\al> K}\right) n_\beta dxds\nonumber\\
\leq &\sum_\al \int \Gamma_K(n_\al(t_0))dx+4\sum_{\al,\beta} |b_{\al\beta}| \int_{t_0}^{t_1}\int n_\al\widetilde{\log}_K n_\al n_\beta dxds.\nonumber
\end{align}
Now picking a constant $A \in[e, K]$, we estimate the last term on the right hand side of \eqref{H2evolution_1} using GNS inequality as follows:
\begin{align}
\sum_{\al,\beta}&|b_{\al\beta}|\int n_\al\widetilde{\log }_K{n_\al}n_\beta dx\nonumber\\
=&\sum_{\al,\beta}|b_{\al\beta}|\left(\int n_\al\widetilde{\log}_Kn_\al n_\beta \mathbf{1}_{n_\beta\geq A}dx+\int n_\al\widetilde{\log }_Kn_\al n_\beta\mathbf{1}_{n_\beta\leq A} dx\right)\nonumber\\
\leq& \sum_{\al,\beta}|b_{\al\beta}|\left(\int\frac{(n_\al \widetilde{\log}_K n_\al)( n_\beta \widetilde{\log}_K n_\beta)}{\log A} dx+A\int n_\al \widetilde{\log}_K n_\al dx\right)\nonumber\\
\leq& 2\max_\al\left(\sum_\beta |b_{\al\beta}|\right)\sum_{\al}\left(\frac{1}{\log A}\int\bigg(\sqrt{n_\al \widetilde{\log}_K n_\al}\bigg)^4 dx+A(M_\al+S_+[n_\al])\right)\label{H2evolution_2}\\
\leq & 2C_{GNS}^2\max_\al\left(\sum_\beta |b_{\al\beta}|\right)\times\sum_{\al}\Bigg(A(M_\al+S_+[n_\al])\nonumber \\
& \hspace*{2.2cm}+\frac{1}{\log A}\Big(\int n_\al \widetilde{\log}_K n_\al dx\Big)\cdot \Big(\int\bigg|\na\sqrt{n_\al\widetilde{\log}_K n_\al}\bigg|^2dx\Big) \Bigg)\nonumber\\
\leq & 2C_{GNS}^2\max_\al\left(\sum_\beta |b_{\al\beta}|\right)\times
 \sum_{\al}\Bigg (A(M_\al+S_+[n_\al])\nonumber \\
& \hspace*{2.2cm}+\frac{1}{\log A}\Big(M_\al+S_+[n_\al]\Big)\cdot\Big(\int\frac{|\na(n_\al)|^2}{n_\al}\widetilde{\log}_Kn_\al\mathbf{1}_{n_\al\geq e}dx+F[n_\al]\Big)\Bigg).\nonumber
\end{align}
Now combining \eqref{H2evolution_1} and \eqref{H2evolution_2} and taking $K$ then $A$ large, we have the estimate
\begin{align*}
\sum_\al \int \Gamma_K(n_\al(t_1))dx\leq & \sum_\al \int \Gamma_K(n_\al(t_0))dx\\ &+2TC_{GNS}\max_\al\left(\sum_\beta |b_{\al\beta}|\right)\sum_\al A(M_\al+S_+[n_\al]) +4\sum_\al\int_{t_0}^{t_1}F[n_\al]ds.
\end{align*}
By the Lemma \ref{Lem:Renormalized relation}, we have that the estimate \eqref{EQ:H2entropy_estimate} holds with the constant $C_{S_2}$ depending on $T, \mathcal{A}_T$ and $\sup_{0\leq t\leq T}V[\mathbf{n}(t)]$.

\textbf{Step 2:} As in \cite{EganaMischler16}, we define the following renormalization function $\gamma_K$, $K\geq e$ approximating $(\cdot )^p$:
\begin{align}\label{betaK_Lp}
\gamma_K(u):=\left\{\begin{array}{rr}\ba\displaystyle&\frac{u^p}{p},\quad u\leq K;\\
\displaystyle&\frac{K^{p-1}}{\log K}(u\log u-u) -\frac{p-1}{p}K^p+\frac{K^p}{\log K}, \quad u>K.\ea\end{array}\right.
\end{align}
We can estimate the $|\gamma_K'(u)u-\gamma_K(u)|$ as follows
\begin{align*}
{|\gamma_K'(u)u-\gamma_K(u)|\leq \frac{p-1}{p}u^p\mathbf{1}_{u\leq K}+2K^{p-1}u\mathbf{1}_{u>K}}.
\end{align*}
Applying this estimate in the \eqref{Renormalized relation} yields
\begin{align}
\sum_\al\int \gamma_K(n_\al(t_1))dx+&\sum_\al\frac{4(p-1)}{p^2}\int_{t_0}^{t_1}\int |\na (n_\al^{p/2})|^2\mathbf{1}_{n_\al\leq K}dxds\nonumber \\
& +\frac{K^{p-1}}{\log K}\sum_\al\int_{t_0}^{t_1}\int \frac{|\na n_\al|^2}{n_\al}\mathbf{1}_{n_\al\geq K}dxds\nonumber\\
\leq & \sum_\al\int\gamma_K(n_\al(t_0))dx+\frac{p-1}{p}\sum_{\al,\beta}|b_{\al\beta}|\int_{t_0}^{t_1}\int n_\al^p\mathbf{1}_{n_\al\leq K}n_\beta dxds\label{Lp_proof_1}\\
& +2K^{p-1}\sum_{\al,\beta}|b_{\al\beta}|\int_{t_0}^{t_1}\int n_\al \mathbf{1}_{n_\al> K}n_\beta dxds\nonumber\\
=: &\sum_\al\int\gamma_K(n_\al(t_0))dx+T_1+T_2.\nonumber 
\end{align}
For the second term $T_1$ on the right hand side of \eqref{Lp_proof_1}, we decompose it as follows:
\begin{align}
T_1=&
\frac{p-1}{p}\sum_{\al,\beta}|b_{\al\beta}|\int_{t_0}^{t_1}\int n_\al^p\mathbf{1}_{n_\al\leq K}n_\beta(\mathbf{1}_{n_\beta\leq K}+\mathbf{1}_{n_\beta>K}) dxds\nonumber\\
\leq &\frac{p-1}{p}\max_\al\left(\sum_\beta |b_{\al\beta}|\right)\sum_{\al}\int_{t_0}^{t_1}\int n_\al^{p+1}\mathbf{1}_{n_\al\leq K}dxds\label{Lp_proof_2}\\
 & +\frac{K^{p-1}(p-1)}{p}\sum_{\al,\beta}|b_{\al\beta}|\int_{t_0}^{t_1}\int n_\beta^2\mathbf{1}_{n_\beta>K}dxds\nonumber\\
=:&T_{11}+T_{12}\nonumber 
\end{align}
The treatment of the $T_{11}$ term is similar to the corresponding one in the proof of Lemma \ref{L_p_bound_n_al}. It can be estimated using the Gagliardo-Nirenberg-Sobolev inequality, Chebyshev inequality and a classical vertical truncation technique with truncation level $A\in (0,K)$ as follows:
\begin{align}
T_{11}=&\frac{p-1}{p}\max_\al\left(\sum_\beta |b_{\al\beta}|\right)\sum_{\al}\int_{t_0}^{t_1}\int \bigg(\min\{n_\al, A\}+(n_\al\mathbf{1}_{n_\al\leq K}-A)_+\bigg)^{p+1}dxds\nonumber\\
\leq &\frac{2^p(p-1)}{p}A^p\max_\al\left(\sum_\beta |b_{\al\beta}|\right)\sum_{\al}M_\al(t_1-t_0)\nonumber \\
& +\frac{2^p(p-1)}{p}\max_\al\left(\sum_\beta |b_{\al\beta}|\right)\sum_{\al}\iint \bigg(n_\al\mathbf{1}_{n_\al\leq K}-A\bigg)_+^{p+1}dxds\label{T11}\\
\leq & \frac{2^p(p-1)}{p}A^p\max_\al\left(\sum_\beta |b_{\al\beta}|\right)\sum_{\al}M_\al(t_1-t_0)\nonumber \\
& +\max_\al\left(\sum_\beta |b_{\al\beta}|\right)\sum_{\al}\frac{C_{GNS}S_+[n_\al]}{\log A}\iint|\na(n_\al^{p/2})|^2\mathbf{1}_{n_\al\leq K}dxds.\nonumber
\end{align} Here we can see that if we choose $K$ then $A$ large enough, the second term can be absorbed by the dissipative term on the left hand side of \eqref{Lp_proof_1}. The second term $T_{12}$ in \eqref{Lp_proof_2} has a different flavor. Here the improved integrability of the solution  \eqref{EQ:H2entropy_estimate} is applied to gain extra smallness on this nonlinear term. Similar to the paper \cite{EganaMischler16}, we apply the bound \eqref{EQ:H2entropy_estimate}, the Sobolev inequality and Cauchy-Schwarz inequality to estimate the $T_{12}$ term in \eqref{Lp_proof_2} as follows:
\begin{align}
T_{12}\leq& \frac{4K^{p-1}(p-1)}{p}\sum_{\al,\beta}|b_{\al\beta}|\int_{t_0}^{t_1}\int \left(n_\beta-{K}/{2}\right)_+^2dxds\nonumber\\
\leq&\frac{4C_SK^{p-1}(p-1)}{p}\sum_{\al,\beta}|b_{\al\beta}|\int_{t_0}^{t_1}\bigg(\int |\na(n_\beta-{K}/{2})_+|dx\bigg)^2ds\nonumber\\
\leq&\frac{4C_SK^{p-1}(p-1)}{p}\sum_{\al,\beta}|b_{\al\beta}|\int_{t_0}^{t_1}\left(\int \frac{|\na n_\beta|^2}{n_\beta}(\mathbf{1}_{K\geq n_\beta\geq K/{2}}+\mathbf{1}_{n_\beta> K})dx\right)\left(\int n_\beta\mathbf{1}_{n_\beta\geq{K}/{2}}\right)ds\nonumber\\
\leq&\sum_{\al,\beta}|b_{\al\beta}|\frac{32(p-1) C_S\sup_{t_0\leq t\leq t_1}S_2[n_\beta(t)]}{p(\log K)^2}\bigg(\frac{2^{p+1}}{p^2}\int_{t_0}^{t_1}\int |\na(n_\beta^{p/2})|^2\mathbf{1}_{K/2\leq n_\beta\leq K}dxds\nonumber\\
&+K^{p-1}\int_{t_0}^{t_1}\int \frac{|\na n_\beta|^2}{n_\beta}\mathbf{1}_{n_\beta>K}dxds\bigg).\label{T12}
\end{align}
Since $S_2$ is bounded on the time interval $[t_0,t_1]$ \eqref{EQ:H2entropy_estimate}, if $K$ is large enough, these terms can be absorbed by the left hand side of \eqref{Lp_proof_1}.

For the last term $T_2$ on the right hand side of \eqref{Lp_proof_1}, applying the symmetry of the matrix $\mathbf{B}$ \eqref{b symmetric}, H\"older inequality and the Young's inequality,  we can estimate it as follows
\begin{align}
T_2=&2K^{p-1}\sum_{\al,\beta}\int_{t_0}^{t_1}\int n_\al \mathbf{1}_{n_\al> K}|b_{\al\beta}|n_\beta(\mathbf{1}_{n_\beta> K}+\mathbf{1}_{n_\beta\leq K})dxds\nonumber\\
\leq&4K^{p-1}\max_\al\left(\sum_\beta |b_{\al\beta}|\right)\sum_{\al}\int_{t_0}^{t_1}\int n_\al^2\mathbf{1}_{n_\al>K} dxds.\label{T2}
\end{align}
Now they are similar to the $T_{12}$ term in \eqref{Lp_proof_2} and we skip the treatment for the sake of brevity.

Combining the estimates \eqref{Lp_proof_2}, \eqref{T11} and \eqref{T2}, we have from \eqref{Lp_proof_1} that
\begin{align*}
\sum_\al&\int \gamma_K(n_\al(t_1))dx+\sum_\al\frac{2(p-1)}{p^2}\int_{t_0}^{t_1}\int |\na (n_\al^{p/2})|^2\mathbf{1}_{n_\al\leq K}dxds\nonumber\\
\leq &\sum_\al\int\gamma_K(n_\al(t_0))dx+2^p A^p \max_\al\left(\sum_\beta |b_{\al\beta}|\right)\sum_{\al}M_\al T.
\end{align*}
Now we can take $A$ fixed and $K$ to infinity to complete the proof of the lemma.

\end{proof}
Next, arguing along the lines of  \cite{EganaMischler16}, we end up with the conclusion that free energy solutions are classical solution for all positive time. We quote 
\begin{lem}[\cite{EganaMischler16}]\label{lem:smoothness}
Any physically relevant free energy solutions $(n_\al)_{\al\in \mathcal{I}}$ to \eqref{EQ:KS_multiple_groups} are smooth for any strictly positive time, i.e.,
\begin{equation}
n_\al\in C^\infty((\delta,T_\star)\times \rr^2),\quad \forall \delta>0.
\end{equation}
\end{lem}
Moreover, we have the following lower semicontinuity of the free energy functional.
\begin{lem}[\cite{EganaMischler16}]\label{E_Lower_semi_continuous}
Consider any bounded sequences $(n_{\al,k})_{\al\in \mathcal{I}}$ of nonnegative functions in $L^1_+(\rr^2)$ with finite second moment $\sum_\al\int n_{\al,k}|x|^2dx<\infty$. Assume that $\{n_{\al,k}\}_{k=1}^\infty$ has the same subcritical masses as $n_\al$, i.e., $|n_{\al, k}|_1=M_\al$, $\forall \al \in \mathcal{I},$ $\forall k\in \mathbb{N}$. If there exists a constant $C$ such that the free energy $E[(n_{\al,k})_{\al\in\mathcal{I}}]$ is uniformly bounded in $k$, i.e., $\sup_k E[(n_{\al,k})_{\al\in \mathcal{I}}]\leq C<\infty$, and $\{n_{\al,k}\}_{k=1}^\infty$ converges to $n_\al$ in $\mathcal{D}'(\rr^2)$ for all $\al\in \mathcal{I}$, there holds
\begin{equation}
n_\al\in L_+^1(\rr^2), \quad \int n_\al|x|^2dx<\infty, \quad\forall\al\in \mathcal{I} \quad \text{and}\quad E[(n_\al)_{\al\in \mathcal{I}}]\leq \liminf_{k\ra\infty} E[(n_{\al,k})_{\al\in \mathcal{I}}].
\end{equation}
\end{lem}

\noindent
Equipped with lemma \ref{lem:smoothness} and \ref{E_Lower_semi_continuous} we turn to the following.
\begin{proof}[Proof of Theorem \ref{Theorem_of_smoothness_of_solutions}]
The smoothness of the solutions is proved in Lemma \ref{lem:smoothness}. The proof of the equality in \eqref{Free_energy_dissipation} is similar to the one in \cite{EganaMischler16}. For the sake of completeness, we detailed the proof as follows. 

Since the solution $n_\al,\al\in \mathcal{I}$ is smooth for all positive time, the following equality holds for all $t_n>0$, where $t_n\rightarrow 0^+$:
\begin{equation}
E[\mathbf{n}(t)]=E[\mathbf{n}(t_n)]+\sum_\al\int_{t_n}^tn_\al|\na \log n_\al -\na c_\al|^2dxds.
\end{equation}
Combining this with the Lebesgue dominated convergence theorem, the lower semi-continuity of the functional $E$ proven in the last lemma and the fact that $\mathbf{n}(t_n)$ converges to $\mathbf{n}_0$ weakly in $\mathcal{D}'(\rr^2)$, we have that
\begin{align}
E[\mathbf{n}_0]\leq& \liminf_{n\rightarrow 0}E[\mathbf{n}(t_n)]\leq \lim\left(E[\mathbf{n}(t)]+\sum_\al\int_{t_n}^t n_\al|\na\log n_\al-\na c_\al|^2dxds\right)\nonumber\\
=&E[\mathbf{n}(t)]+\sum_\al\int_{0}^t n_\al|\na\log n_\al-\na c_\al|^2dxds.
\end{align}
Recalling the definition of the free energy solution, the proof of the free energy dissipation equality is completed.
\end{proof}
\section{Uniqueness of the free energy solutions}
After proving the smoothness theorem for the system \eqref{EQ:KS_multiple_groups}, we are ready to prove the uniqueness of the physically relevant free energy solutions $(n_\al)_{\al\in \mathcal{I}}$. To estimate the deviation between two solutions on a small time interval, some smallness estimates are needed. The following lemma provides the functional space where we could seek for smallness.

\begin{lem}
Consider the physically relevant free energy solution $\mathbf{n}$ to the system \eqref{EQ:KS_multiple_groups}. The following holds
\begin{equation}\label{t1/4n4/3_0}
\lim_{t\rightarrow 0^+}t^{1/4}\sum_\al|n_\al(t)|_{4/3}=0.
\end{equation}
\end{lem}
\begin{proof}
The proof is similar to the one in the paper \cite{EganaMischler16}. Before estimating the norm $t^{1/4}|n_\al|_{4/3}$, we collect some estimates which we are going to use. It is enough to consider a short interval $[0,T]\subset [0,T_\star)$. From the assumptions \eqref{Total_second_moment_evolution}, \eqref{A_t_Bound} we have that the positive part of the entropy is bounded 
\begin{align*}
\sum_\al S_+[n_\al(t)]\leq C_{L\log L}<\infty,\quad \forall t\in [0,T].
\end{align*} 
Next we prove the estimate
\begin{equation}\label{t_L22_bound}
\sum_\al|n_\al(t)|_2^2t\leq C_{L^2}(\mathbf{B},\mathbf{M},|\mathcal{I}|, C_{L\log L})<\infty,\quad\forall t\in [0, T].
\end{equation}
Standard $L^2$ energy estimate yields
\begin{equation}
\frac{d}{dt}\sum_\al|n_\al|_2^2+2\sum_\al|\na n_\al|_2^2=\sum_{\al,\beta\in \mathcal{I}}b_{\al\beta}\int n_\al^2n_\beta dx.
\end{equation}
Applying the Nash inequality, Gagliardo-Nirenberg-Sobolev inequality and the vertical truncation technique applied in the proof of Lemma \ref{L_p_bound_n_al}, we estimate the right hand side as follows
\begin{align}
\frac{d}{dt}\sum_\al|n_\al|_2^2\leq&-\sum_\al |\na n_\al|_2^2+\sum_{\al,\beta}|b_{\al\beta}|\cdot|n_\beta|_3^3\nonumber\\
\leq&-\sum_\al|\na n_\al|_2^2+\sum_{\al,\beta}|b_{\al\beta}|\bigg(|n_\beta \mathbf{1}_{n_\beta\leq K}|_3^3+|n_\beta\mathbf{1}_{n_\beta\geq K}|_1^{1/3}|n_\beta\mathbf{1}_{n_\beta\geq K}|_4^{8/3}\bigg)\nonumber\\
\leq&-\sum_\al|\na n_\al|_2^2+\sum_{\al,\beta}|b_{\al\beta}|\bigg(K^2M_\beta+\frac{C_{GNS}\sup_{t\in[0,T]}S_+[\mathbf{n}(t)]^{1/3}}{(\log K) ^{1/3}}|n_\beta|_1^{2/3}|\na n_\beta|_2^{2}\bigg)\nonumber\\
\leq&-\sum_\al\left(1-\sum_\beta|b_{\al\beta}|\frac{C_{GNS}C_{L\log L}^{1/3}}{(\log K)^{1/3}}M_\al^{2/3}\right)|\na n_\al|_2^2+\sum_{\al,\beta}|b_{\al\beta}|K^2M_\beta\nonumber\\
\leq& -\frac{(\sum_\al|n_\al|_2^2)^2}{2C_{N}\max_\al M_\al^2|\mathcal{I}|}+\sum_{\al,\beta}|b_{\al\beta}|K^2M_\beta,\label{t1/4n4/3_0_1}
\end{align}
where $K$ is a large number chosen such that the coefficient of $|\na n_\al|_2^2$ is less than $-1/2$. Now by comparing $|n_\al|_2$ with the solution to the super equation
\begin{equation*}
\frac{d}{dt}f= -\frac{f^2}{2C_{N}\max_\al M_\al^2|\mathcal{I}|}+K^2\sum_{\al,\beta}|b_{\al\beta}| M_\beta,\quad f(0)=\infty,
\end{equation*} we obtain \eqref{t_L22_bound}.

Now we estimate the quantity $t^{1/4}|n_\al(t)|_{4/3}$. By the H\"older's inequality and the boundedness of the entropy, we have that
\begin{align}
\left(t^{1/4}| n_\al|_{4/3}\right)^{4/3}=&t^{1/3}\int n_\al^{4/3}dx\leq \left(\int n_\al(\log^+ n_\al+2)dx\right)^{2/3}\left(t\int n_\al^2(2+\log^+ n_\al)^{-2}dx\right)^{1/3}\nonumber\\
\leq& C(C_{L\log L},\mathbf{M})\left(t\int n_\al^2(2+\log^+ n_\al)^{-2}dx\right)^{1/3}.\label{L4/3}
\end{align}
To estimate the term in the parenthesis, we separate the integral into two parts and use the increasing property of the function $s/(2+\log^+s)^2$, the conservation of mass and \eqref{t_L22_bound} to estimate each piece
\begin{align*}
t\int n_\al^2(2+\log^+ n_\al)^{-2}dx\leq &t\int_{n_\al\leq R}n_\al^2(2+\log^+ n_\al)^{-2}dx+ t\int_{n_\al >R}n_\al^2(2+\log^+n_\al)^{-2}dx\\
\leq & t\frac{R}{(2+\log^+R)^2}\int_{n_\al\leq R}n_\al dx+\frac{t}{(2+\log^+R)^2}\int_{n_\al\geq R} n_\al^2dx\\
\leq & t\frac{MR}{(2+\log^+R)^2}+\frac{C_{L^2}}{(2+\log^+R)^2}.
\end{align*}
Now set $R:=1/t$, we have
\begin{equation}
t\int n_\al^2(2+\log^+ n_\al)^{-2}dx\leq\frac{M+C_{L^2}}{(2+\log^+1/t)^2}\rightarrow 0, \quad t\rightarrow 0_+.
\end{equation}
Combining this with \eqref{L4/3} yields the result.
\end{proof}

Now we prove the Theorem \ref{Theorem_of_uniqueness}. Consider the equation \eqref{EQ:KS_multiple_groups} in the mild form. Since we have smoothness of the free energy solution, we have that the two formulation are equivalent. Suppose that $(n_{\al,1})_{\al\in \mathcal{I}}, (n_{\al,2})_{\al\in\mathcal{I}}$ are two solutions subject to the same initial data $n_{\al 0}$, $\al\in\mathcal{I}$, their difference satisfies:
\begin{align*}
n_{\al,2}(t)-n_{\al,1}(t)=&-\int_0^t e^{(t-s)\de}\na\cdot\left((\na c_{\al,2}(s)-\na c_{\al,1}(s))n_{\al,2}(s)\right)ds\\
&-\int_0^t e^{(t-s)\de}\na \cdot\left(\na c_{\al,1}(s)(n_{\al,2}(s)-n_{\al,1}(s))\right)ds,\quad \forall\al\in\mathcal{I}.
\end{align*}
Define the following quantities:
\begin{align}
Z_{\al,\ell}(t):=&\sup_{0<s\leq t}s^{1/4}|n_{\al,\ell}(s)|_{4/3}, \quad \ell=\{1,2\};\\
\de_{\al}(t):=&\sup_{0<s\leq t}s^{1/4}|n_{\al,2}(s)-n_{\al,1}(s)|_{4/3},\quad \forall \al \in \mathcal{I}.
\end{align}
The estimate \eqref{t1/4n4/3_0} yields that $\lim_{t\ra 0_+}Z_{\al,\ell}(t)=0$. The $\de_\al(t)$ can be further decomposed as follows:
\begin{align}
\de_\al(T)
\leq& \sup_{0\leq t\leq T}t^{1/4}\bigg|{\int_0^t e^{(t-s)\de}\na\cdot((\na c_{\al,2}(s)-\na c_{\al,1}(s))n_{\al,2}(s))ds}\bigg|_{4/3}\nonumber\\
&+\sup_{0\leq t\leq T}t^{1/4}\bigg|{\int_0^t e^{(t-s)\de}\na\cdot (\na c_{\al,1}(s)(n_{\al,2}(s)-n_{\al,1}(s)))ds}\bigg|_{4/3}\nonumber\\
=:&\sup_{0\leq t\leq T}J_{\al,1}(t)+\sup_{0\leq t\leq T}J_{\al,2}(t).\label{J1J2}
\end{align}
Now we estimate the $J_{\al,2}$ term in \eqref{J1J2} using the H\"older inequality, Hardy-Littlewood-Sobolev inequality, Minkowski integral inequality and heat semigroup estimate as follows
\begin{align}
J_{\al,2}(t)\leq &t^{1/4}\int_0^t\frac{C}{(t-s)^{3/4}}|\na c_{\al,1}|_{4}|n_{\al,2}-n_{\al,1}|_{4/3}ds\nonumber\\
\leq&\int_0^t C\frac{t^{1/4}}{s^{1/2}(t-s)^{3/4}}ds\sum_{\beta\in \mathcal{I}}|b_{\al\beta}|Z_{\beta,1}(t)\de_\al(t)\nonumber\\
\leq&C\sum_{\beta\in \mathcal{I}}|b_{\al\beta}|Z_{\beta,1}(t)\de_\al(t)\label{J2}.
\end{align}
Similarly, we can estimate the $J_{\al,1}$ term as follows:
\begin{align}
J_{\al,1}(t)\leq C\sum_\beta |b_{\al\beta}|\de_\beta(t)Z_{\al,2}(t).\label{J1}
\end{align}
Combining \eqref{J1J2}, \eqref{J1}, \eqref{J2} and symmetry of $\mathbf{B}$ \eqref{b symmetric}, we have that
\begin{align*}
\sum_\al\de_\al(T)
\lesssim&\sum_{\al,\beta}|b_{\al\beta}|\sup_{0\leq t \leq T} Z_{\beta,1}(t)\de_\al(t)+\sum_{\al,\beta}|b_{\al\beta}|\sup_{0\leq t\leq T}\de_\beta(t)Z_{\al,2}(t)\\
\lesssim &\sum_{\al,\beta}| b_{\al\beta}|\sup_{0\leq t\leq T}\de_\al(t)(Z_{\beta,1}(t)+Z_{\beta,2}(t))\\
\lesssim &\max_{\al,\beta}|b_{\al\beta}|\sum_{\al}\de_\al(T)\left(\sum_{\beta}\sum_{\ell=1}^2 Z_{\beta,\ell}(T)\right).
\end{align*}
Now since $Z_{\beta,\ell} (t)$ approaches zero as time approaches $0_+$ \eqref{t1/4n4/3_0}, there exists a small time $T'$ such that
\begin{equation}
\sum_\al\de_\al(T')\leq \frac{1}{2}\sum_\al\de_\al(T'),\quad T'\in [0,T].
\end{equation}
So we have $\sum_\al\de_\al\equiv 0$, $\forall t\in[0,T']$. Now the uniqueness follows if we iterate this argument.
\section{Long time behavior of the free energy solutions}
In this section, we studied the long time behavior of the multi-species PKS system \eqref{EQ:KS_multiple_groups}. Since the solution becomes instantly smooth, we could assume that the initial data $n_{\al 0}$ is $C^\infty\cap L^1$ for all $\al\in \mathcal{I}$. We rewrite the equation \eqref{EQ:KS_multiple_groups} in the self-similar variables
\begin{align*}
X:=\frac{x}{R(t)},\quad \tau:=\log R(t),\quad R(t):=\sqrt{1+2t}.
\end{align*}
We define the solutions $N_\alpha, C_\al$ in the self-similar variables:
\begin{align}
n_\al(x,t)=\frac{1}{R^2(t)}N_\al(X,\tau),\quad c_\al(x,t)=C_\al(X,\tau).
\end{align}
Rewriting the equation \eqref{EQ:KS_multiple_groups} in the self-similar variables, we obtain that the $N_\al$, $C_\al$ satisfy the following equations subject to initial data $N_\al(X,\tau=0)(n_{\al 0}(X)$, $\forall\al \in \mathcal{I}$:
\begin{align}\label{EQ:KS_multiple_groups_self_similar_variable}\left\{\begin{array}{rr}\ba
\pa_\tau N_\al=&\de N_\al+\na\cdot(XN_\al)-\na\cdot(\na C_\al N_\al),\\
-\de C_\al=&\sum_{\beta\in\mathcal{I}}b_{\al\beta} N_\beta.
\ea\end{array}\right.
\end{align}
In order to prove Theorem \ref{Theorem_of_long_time_decay}, we show that the solution $N_\al$ to the equation \eqref{EQ:KS_multiple_groups_self_similar_variable} is uniformly bounded in time. This is due to the fact that the $L^2(dx)$ norm of solutions $n_\al$ to the original problem and the $L^2(dX)$ norm of the solutions $N_\al$ to the equation \eqref{EQ:KS_multiple_groups_self_similar_variable} have the following relation:
\begin{equation}
|n_\al|_{L^2(dx)}^2=\frac{|N_\al|_{L^2(dX)}^2}{R^2( t)}=\frac{|N_\al|_{L^2(dX)}^2}{1+2t}.
\end{equation}
Therefore any uniform in time bound of $|N_\al|_{L^2(dX)}$ can be translated to decay of $|n_\al|_{L^2(dx)}$. We decompose our proof into several lemmas. First we show that the second moment of the solutions are uniformly bounded in time.
\begin{lem}
Consider the solutions $N_\al$, $\al\in\mathcal{I}$ to the equation \eqref{EQ:KS_multiple_groups_self_similar_variable}. The total second moment is uniformly bounded in time, i.e.,
\begin{equation}\label{Second_moment_uniform_bound_self_similar}
\sum_{\al\in\mathcal{I}} \int N_\al(X,\tau)|X|^2dX\leq C_{V,R}<\infty,\quad \forall \tau\in[0,\infty).
\end{equation}
\end{lem}
\begin{proof}
Similar to the proof of \eqref{time_evol_V}, we calculate the time evolution of the second moment
\begin{align*}
\frac{d}{d\tau}\sum_\al \int N_\al|X|^2dX=&-2\sum_\al\int N_\al |X|^2dX+\bigg(\sum_\al 4 M_\al\bigg)\bigg(1-\frac{Q_{\mathbf{B},\mathbf{M}}[\mathcal{I}]}{8\pi}\bigg).
\end{align*}
Now we see that the total second moment is bounded
\begin{align*}
\sum_\al\int N_\al|X|^2dX\leq \max\left\{\frac{1}{2}\bigg(\sum_\al 4 M_\al\bigg)\bigg(1-\frac{Q_{\mathbf{B},\mathbf{M}}[\mathcal{I}]}{8\pi}\bigg),\sum_\al\int (N_\al)_0|X|^2dX\right\}.
\end{align*}
\end{proof}
Similar to the proof of the estimate \eqref{Free_energy_dissipation}, we can show that the equation \eqref{EQ:KS_multiple_groups_self_similar_variable} has the following decreasing free energy for $\forall \tau\geq 0$:
\begin{align*}
E_{R}[\mathbf{N}(\tau)]=\sum_{\al\in\mathcal{I}}\int N_\al\log N_\al dX+&\sum_{\al,\beta\in\mathcal{I}}\frac{b_{\al\beta}}{4\pi}\iint \log|X-Y|N_\al(X)N_\beta(Y)dXdY\\
& +\frac{1}{2}\sum_{\al\in\mathcal{I}}\int N_\al |X|^2dX\leq E_{R}[\mathbf{N}_0].
\end{align*}
Now we apply the log-HLS inequality \eqref{log_HLS_for_system} to get a bound for the entropy, $S_R[\mathbf{N}]=\sum_\al\int N_\al\log N_\al dX$, obtaining
\begin{align*}
E_R[\mathbf{N}_0]\geq &E_R[\mathbf{N}]\nonumber\\
\geq&\sum_{\al\in \mathcal{I}}\int N_\al\log N_\al dX+\sum_{\al,\beta\in \mathcal{I}}\frac{(b_{\al\beta})_+}{4\pi}\int N_\al(X)\log|X-Y|N_\beta(Y)dXdY\nonumber\\
&-\sum_{\al,\beta}\frac{(b_{\al\beta})_-}{4\pi}\iint_{|X-Y|\geq 1}N_\al(X)\log|X-Y|N_\beta(Y)dXdY+\frac{1}{2}\int N_\al|X|^2dX\nonumber\\
=&(1-\theta)\sum_{\al\in\mathcal{I}}\int N_\al\log N_\al dX\nonumber\\
 & +\theta\left(\sum_{\al\in\mathcal{I}} \int N_\al\log N_\al dx+\frac{1}{4\pi}\sum_{\al,\beta\in \mathcal{I}}\frac{(b_{\al\beta})_+}{\theta}\iint N_\al(X)\log |X-Y| N_\beta(Y)dXdY\right)\nonumber\\
&-\sum_{\al,\beta}\frac{(b_{\al\beta})_-}{4\pi}(M_\al V_\beta+M_\beta V_\al)+\frac{1}{2}\int N_\al|X|^2dX\nonumber\\
\geq&(1-\theta)\sum_{\al\in\mathcal{I}}\int N_\al\log^+ N_\al dX-(1-\theta)\int N_\al\log^-N_\al dX-\theta C_{lHLS}(\mathbf{B},\mathbf{M})\\
&-\sum_{\al,\beta}\frac{(b_{\al\beta})_-}{4\pi}(M_\al V_\beta+M_\beta V_\al)+\frac{1}{2}\int N_\al|X|^2dX.
\end{align*}
Here the $\theta\in(0,1)$ is chosen as in the proof of Proposition \ref{Prop_1}. Now since the second moment is bounded for all time \eqref{Second_moment_uniform_bound_self_similar}, we have that $C_{lHLS}<\infty$ and the negative part of the entropy is uniformly bounded in time, i.e., $\displaystyle \int N_\al(X,\tau)\log^- N_\al(X,\tau) dX<C<\infty$ for $\forall \tau\in[0,\infty)$, which in term yields that
\begin{equation}\label{NlogN_bound}
\sum_{\al\in\mathcal{I}}\int N_\al(X,\tau)\log^+ N_\al(X,\tau) dX<C_{L\log L,R}<\infty,\quad \forall\tau \in[0,\infty).
\end{equation}
Once the positive part of the entropy is bounded, we estimate the time evolution $\sum_\al|(N_\al-K)_+|_2^2$ as in the proof Lemma \ref{L_p_bound_n_al}
\begin{align*}
\frac{1}{2}\frac{d}{dt}\sum_\al|(N_\al-K)_+|_2^2\leq&\left(-{3}+\eta(K)\max_\al\left(\sum_\beta |b_{\al\beta}|\right)C_{GNS}\right)\sum_\al \int |\na (N_\al-K)_+|^2dX\\
&+C(K,\mathbf{B},\mathbf{M})|(N_\al-K)_+|_2^2+C(K,\mathbf{B},\mathbf{M}),
\end{align*}
where $\eta(K)\leq \frac{C_{L\log L,R}}{\log K}$ is made small enough. Now we choose the $K$ large enough and apply the Nash inequality to get
\begin{align*}
\frac{d}{dt}\sum_\al|(N_\al-K)_+|_2^2\leq & -\frac{(\sum_\al|(N_\al-K)_+|_2^2)^2}{C_N\sum_\al|(N_\al-K)_+|_1^2|\mathcal{I}|}\\
& +C(K,\mathbf{B},\mathbf{M})\sum_\al|(N_\al-K)_+|_2^2+C(K,\mathbf{B},\mathbf{M}).
\end{align*}
Since $|(N_\al-K)_+|_1\leq|N_\al|_1=M_\al<\infty$, we have that $\sum_\al|(N_\al-K)_+|_2\leq C_{L^2,R}<\infty$ for $\forall \tau\in[0,\infty)$. This completes the proof of Theorem \ref{Theorem_of_long_time_decay}.

\section{Multi-species PKS subject to non-symmetric coupling arrays}
\subsection{Symmetrizable case}
In general,  the chemical generation coefficient matrix $\mathbf{B}$ is non-symmetric. This introduces new challenges in the analysis. We will not cover the general situation in this paper. However, in certain cases, one can symmetrize the system. First recall the $sign$ function:
\begin{equation}
sign (f)=\left\{\begin{array}{rrr} 1, \quad f>0;\\
0,\quad f=0;\\
-1,\quad f<0.
\end{array}\right.
\end{equation}
If $sign (b_{\al\beta})=sign (b_{\beta\al})$ and the matrix $\mathbf B$ is three diagonal, i.e., $b_{\al\beta}\neq 0$ only if $|\al-\beta|\leq 1$, the system can always be symmetrized. Specifically, all the two species models with $sign( b_{12})=sign (b_{21})$ are symmetrizable. To show the method, we consider system \eqref{EQ:KS_multiple_groups} subject to general 3-by-3 matrix
\begin{align*}
\pa_t {n}_\al&+\sum_{\beta\in \{1,2,3\}}\na \cdot(b_{\al\beta} (-\na\de^{-1})n_\beta n_\al)=\de n_\al,\quad \al\in\{1,2,3\},\\
\mathbf B=&\left[\begin{array}{rrr}
b_{11},\quad b_{12},\quad b_{13}\\
b_{21},\quad b_{22},\quad b_{23}\\
b_{31},\quad b_{32},\quad b_{33}
\end{array}\right],\quad sign (b_{\al\beta})=sign (b_{\beta\al}),\quad b_{13}=b_{31}=0.
\end{align*}
First we can multiply the equation of $n_2$ by $b_{12}/b_{21}$ and redefine $\tilde{n}_2:=\frac{b_{12}}{b_{21}}n_2$ to obtain
\begin{align*}
\pa_t {n}_1+ &\na \cdot(b_{11} (-\na\de^{-1})n_1n_1+ b_{21} (-\na\de^{-1})\tilde{n}_2 n_1)=\de n_1;\\
\pa_t \tilde{n}_2+ &\na \cdot\left( b_{21} (-\na\de^{-1}) n_1\tilde{n}_2+ \frac{b_{21}b_{22}}{b_{12}} (-\na\de^{-1})\tilde{n}_2\tilde{n}_2+ b_{23} (-\na\de^{-1})n_3\tilde{n}_2\right)=\de\tilde{n}_2.
\end{align*}
Now we can do the same trick on the third equation by multiplying it by $\frac{b_{12}b_{23}}{b_{32}b_{21}}$ and redefine $\tilde{n}_3:=\frac{b_{12}b_{23}n_3}{b_{32}b_{21}}$, we obtain  that
\begin{align*}
\pa_t \tilde{n}_2+ &\na \cdot\left( b_{21} (-\na\de^{-1}) n_1\tilde{n}_2+ \frac{b_{21}b_{22}}{b_{12}} (-\na\de^{-1})\tilde{n}_2\tilde{n}_2+ \frac{b_{32}b_{21}}{b_{12}} (-\na\de^{-1})\tilde{n}_3\tilde{n}_2\right)=\de\tilde{n}_2,\\
\pa_t \tilde{n}_3+ &\na \cdot\left( \frac{b_{32} b_{21}}{b_{12}}(-\na\de^{-1}) \tilde{n}_2\tilde{n}_3+ \frac{b_{32}b_{21}b_{33}}{b_{12}b_{23}} (-\na\de^{-1})\tilde{n}_3\tilde{n}_3\right)=\de\tilde{n}_3.
\end{align*}
Now we see that the new coefficient matrix is symmetric. For general tridiagonal matrix with $sign(b_{\al\beta})=sign(b_{\beta\al})$, the symmetrization is similar.

\begin{rmk}
This three diagonal chemical generation matrices $\mathbf{B}$'s correspond to the fact that there exists a hierarchical structure in the community, in which one species only communicates to their direct neighbors. 
\end{rmk}

\subsection{Essentially dissipative case}
In this section, we prove Theorem \ref{Theorem_essentially_negative_B}.
\begin{proof}[Proof of Theorem \ref{Theorem_essentially_negative_B}] First note that if $\mathcal{I}^{(|\mathcal{I}|)}=\mathcal{I}$, then $\mathcal{I}^{(0)}$ is not an empty set. Otherwise one obtain that $\mathcal{I}^{(|\mathcal{I}|)}$ is an empty set, which is a contradiction. We prove that $\sum_\al|n_{\al}(t)|_{L_t^\infty(0,\infty;H_x^s)}\leq C_{H^s}<\infty$.

First we prove the $L^\infty$ bound of the $n_{\al}$'s. We pick all the $\al^0\in \mathcal{I}^{(0)}$, and calculate the time evolution of the $|n_{\al^0}|_{2p}^{2p}$, $\forall p\in [1,\infty)$ utilising the fact that $b_{\al^0\beta}\leq 0$ for all $\beta\in \mathcal{I}$

\begin{align}
\frac{1}{2p}\frac{d}{dt}|n_{\al^0}|_{2p}^{2p}=&-\frac{2p-1}{p^2}|\na (n_{\al^0})^p|_2^2-\frac{2p-1}{2p}\int n_{\al^0}^{2p}\de c_{\al^0}dx\nonumber\\
=&-\frac{2p-1}{p^2}|\na n_{\al^0}|_2^2+\frac{2p-1}{2p}\sum_{\beta\in\mathcal{I}}b_{\al^0\beta}\int n_{\al^0}^{2p} n_{\beta}dx\leq0.
\label{n_al_0_L_2_bound}
\end{align}
As a result, for any $p\in [1,\infty)$, $|n_{\al^0}|_{2p}\leq |(n_{\al^0})_0|_{2p}$. Since the initial data is in $L^1\cap L^\infty$, we have that $\max_{\al^0\in\mathcal{I}^{(0)}}|n_{\al^0}|_{L_t\infty(0,\infty;L_x^\infty)}\leq C_{\mathcal{I}^{(0)}}<\infty$. Next we look at all the $\al^1$'s in the set $\mathcal{I}^{(1)}$. Calculating the time evolution of the $L^{2p}$ norm using the Nash inequality , we have that
\begin{align*}
\frac{1}{2p}\frac{d}{dt} |(n_{\al^1})^p|_{2}^{2}\leq&-\frac{2p-1}{p^2}|\na (n_{\al^1})^p|_2^2+\frac{2p-1}{2p}\sum_{\beta\in \mathcal{I}^{(0)}}b_{\al^1\beta}\int n_\beta n_{\al^1}^{2p}\\
\leq&-\frac{2p-1}{p^2}\frac{|(n_{\al^1})^p|_2^4}{C_N|(n_{\al^1})^p|_1^2}+\frac{2p-1}{2p}\sum_{\beta\in \mathcal{I}^{(0)}}b_{\al^1\beta}|n_\beta|_\infty|(n_{\al^1})^p|_2^2.
\end{align*}
Since $|n_{\beta}|_\infty<C_{\mathcal{I}^{(0)}}<\infty$, $\forall \beta\in\mathcal{I}^{(0)}$, we have that
\begin{align}
\sup_{t\in [0,\infty)}|n_{\al^1}|_{2p}^{2p}\leq \max\{p C_N\sup_{t\in[0,\infty)}|n_{\al^1}|_{p}^{2p}\sum_{\beta\in\mathcal{I}^{(0)}}|b_{\al^1\beta}| C_{\mathcal{I}^{(0)}}, |(n_{\al^1})_0|_{2p}^{2p}\}.
\end{align}
Since $|n_{\al^1}|_{L^1}=M_{\al^1}<\infty$ and $|(n_{\al^1})_0|_{L^\infty}<\infty$, by the Moser-Alikakos iteration, we have that $|n_{\al^1}|_\infty\leq C_{\mathcal{I}^{(1)}}<\infty$. By the same argument, we have that
\begin{align}
 \sup_{t\in[0,\infty)}|n_{\al}(t)|_\infty\leq C_{\infty}<\infty, \quad \forall \al \in \mathcal{I}^{(|\mathcal{I}|)}
\end{align}
Since $\mathbf{B}$ is essentially dissipative, $\mathcal{I}^{(|\mathcal{I}|)}=\mathcal{I}$, we have that $|n_{\al}|_{L_t^\infty(0,\infty;L_x^\infty)}\leq C_\infty$ for all $\al \in \mathcal{I}$.

Next we estimate the $H^s$ ($2\leq s\in \mathbb{N}$) norms of the solutions. Assume that we have already obtained the $H^{s-1}$ estimate, i.e.,
\begin{equation}
|n_{\al}(t)|_{H^{s-1}}\leq C_{H^{s-1}}<\infty,\quad\forall t\in [0,\infty).
\end{equation}
We estimate the time evolution of $\sum_\al|\na^s n_\al|_2^2$ using the GNS inequality and HLS inequality as follows:
\begin{align*}
\frac{d}{dt}\sum_\al|\na^s n_{\al}|_2^2\leq &-\sum_\al |\na ^{s+1} n_\al|_2^2 +\sum_\al |\na c_\al|_\infty^2|\na^s n_\al|_2^2+\sum_\al \sum_{\ell=2}^{s+1}|\na^{\ell}c_\al|_4^2|\na^{s+1-\ell}n_{\al}|_4^2\\
\lesssim & -\sum_\al |\na ^{s+1} n_\al|_2^2 +\sum_{\al,\beta }|b_{\al\beta}|(M_\beta^2+C_\infty^2)|\na ^s n_\al|_2^2 \\
& +\sum_{\al,\beta}\sum_{\ell=2}^{s+1}|b_{\al\beta}|\cdot|\na^{\ell-1}n_\beta|_{4/3}^2|\na ^{s+1-\ell} n_\al|_4^2\\
\lesssim &-\sum_\al \frac{|\na^s n_\al|_2^{2+2/s}}{C_{GNS}|n_\al|_2^{2/s}}+\sum_\al |\na ^s n_\al|_2^2+\sum_\al|n_\al|_2^2.
\end{align*}
Since $\sum_\al|n_\al|_{L_t^\infty(0,\infty;L^2_x)}\leq C_\infty+\sum_\al M_\al$, we have that 
\[
\sum_\al |\na^s n_\al(t)|_{2}\leq C_{H^s}(C_\infty, \sum_\al|\na ^s n_{\al 0}|_2,\mathbf{M},\mathbf{B})<\infty
\]
 for all $t\in [0,\infty)$. This completes the proof of the theorem.
\end{proof}
We conclude with a remark concerning the long time behavior of the solutions. We can rewrite the equation \eqref{EQ:KS_multiple_groups} in the self-similar variables as in Section 6 \eqref{EQ:KS_multiple_groups_self_similar_variable}. Applying similar techniques from the proof of Theorem \ref{Theorem_essentially_negative_B} yields that the solutions $\mathbf{n}$ decay in $L^2$, i.e.,
\begin{equation}
\sum_\al|n_\al(t)|_2^2\leq\frac{C}{1+t},\quad t\in\rr_+.
\end{equation}
Here $C$ is a constant which only depends on the initial data. We sketch the proof as follows. As in Section 6, the goal is to show that $\sum_\al|N_\al|_{L^2(dX)}^2$ is uniformly bounded in time $\tau\in[0,\infty)$. For the sake of simplicity, we use $|\cdot|_p$ to denote $|\cdot|_{L^p(dX)}$. First we estimate the $L^p$ norms of the solutions $n_{\al^0}$, $\al^0\in\mathcal{I}^{(0)}$. Combining standard $L^p$ energy estimates, Nash inequality and the fact that $b_{\al^0\beta}\leq 0$ for all $\beta\in\mathcal{I}$ yields that
\begin{align*}
\frac{1}{2p}\frac{d}{d\tau}|(N_{\al^0})^p|_2^2= &-\frac{2p-1}{p^2}|\na(N_{\al^0})^p|_2^2+\frac{2(2p-1)}{2p}|(N_{\al^0})^p|_2^2+\frac{2p-1}{2p}\sum_\beta b_{\al^0\beta}\int N_{\al^0}^{2p} N_{\beta}dX\\
\leq&-\frac{2p-1}{p^2}\frac{|(N_{\al^0})^p|_2^4}{C_N|(N_{\al^0})^p|_1^2}+\frac{2(2p-1)}{2p}|(N_{\al^0})^p|_2^2.
\end{align*}
This estimates yields that
\begin{align*}
\sup_{\tau\in[0,\infty)}|N_{\al^0}(\tau)|_{2p}^{2p}\leq \max\{p C_N\sup_{\tau\in[0,\infty)}|(N_{\al^0})(\tau)|_p^{2p}, |N_{\al^0}(0)|_{2p}^{2p}\}.
\end{align*}
Since $|N_{\al^0}|_1=M_{\al^0}<\infty$ and $|N_{\al^0}(0)|_{L^1\cap L^\infty}<\infty$, we can apply the Moser-Alikakos iteration to obtain that
\begin{align}
\sup_{\tau\in[0,\infty)}|N_{\al^0}(\tau)|_{L^1\cap L^\infty}\leq C_{\mathcal{I}^{(0)}}<\infty.
\end{align}
Now applying the same iteration technique as the one in the proof of Theorem \ref{Theorem_essentially_negative_B} yields the result.

\begin{rmk}\label{rmk:conjecture}
Direct application of the free energy method yields following general result:\newline
Assume that the matrix $\mathbf{B}$ only has positive entries, i.e., $\mathbf{B}=\mathbf{B}_+$ case. Define the support of a symmetric matrix $C_{m\times m}$ to be the indices of the rows such that there exists non-zero entries in this row, i.e., $supp(C)=\{i\in\{1,2,...,m\}|C_{ij}\neq 0 \text{ for some } j\in\{1,2,...,m\} \}$. If there exists a sequence of positive symmetric matrices $\{\mathbf{B}_\ell\}_{\ell\in \mathcal{L}}$ such that $\sum_{\ell\in\mathcal{L}} \mathbf{B}_\ell =\mathbf{B}$ and 
\[
{Q_{\mathbf{B}_\ell, \mathbf{M}}[\mathcal{J}\cap supp \mathbf{B}_\ell]<Q_{\mathbf{B}_\ell,\mathbf{M}} [\mathcal{I}\cap supp \mathbf{B}_\ell]}<C_\ell<8\pi,
\]
 for all $\emptyset\neq \mathcal{J}\subsetneqq \mathcal{I}$ and  $\forall \ell\in\mathcal{L}$, and 
 \[
 \sum_{\ell\in\mathcal{L}}C_\ell\mathbf{1}_{\al \in supp \mathbf{B}_\ell}<8\pi,
 \]
  for $\forall\al \in\mathcal{I}$, then there exists a global solution. A conjecture is that if this  condition involving the \emph{strict} inequalities fails, namely, if some of the strict inequalities $<$'s are replaced by $>$'s, then there must be a finite time blow-up.
\end{rmk}

\bibliographystyle{abbrv}
\bibliography{nonlocal_eqns,JacobBib,SimingBib}

\def\cprime{$'$}
\begin{thebibliography}{10}

\bibitem{ArenasStevensVelazquez09}
E.~E.~E. Arenas, A.~Stevens, and J.~J.~L. Vel\'azquez.
\newblock Simultaneous finite time blow-up in a two-species model for
  chemotaxis.
\newblock {\em Analysis (Munich), 29}, 2009.

\bibitem{BaiWinkler16}
X.~Bai and M.~Winkler.
\newblock Equilibration in a fully parabolic two-species chemotaxis system with
  competitive kinetics.
\newblock {\em Indiana University Mathematics Journal 65}, 2016.

\bibitem{BedrossianIA10}
J.~Bedrossian.
\newblock Intermediate asymptotics for critical and supercritical aggregation
  equations and {Patlak-Keller-Segel} models.
\newblock {\em Comm. Math. Sci.}, 9:1143--1161, 2011.

\bibitem{BlanchetCarrilloMasmoudi08}
A.~Blanchet, J.~Carrillo, and N.~Masmoudi.
\newblock Infinite time aggregation for the critical {Patlak-Keller-Segel}
  model in $\mathbb{R}^2$.
\newblock {\em Comm. Pure Appl. Math.}, 61:1449--1481, 2008.

\bibitem{BlanchetEJDE06}
A.~Blanchet, J.~Dolbeault, and B.~Perthame.
\newblock Two-dimensional {Keller-Segel} model: Optimal critical mass and
  qualitative properties of the solutions.
\newblock {\em E. J. Diff. Eqn}, 2006(44):1--32, 2006.

\bibitem{CalvezCorriasEbde12}
V.~Calvez, L.~Corrias, and M.~A. Ebde.
\newblock Blow-up, concentration phenomenon and global existence for the
  {Keller-Segel} model in high dimension.
\newblock {\em Communications in Partial Differential Equations Vol. 37 , Iss.
  4}, pages 561--584, 2012.

\bibitem{CarrilloRosado10}
J.~Carrillo and J.~Rosado.
\newblock Uniqueness of bounded solutions to aggregation equations by optimal
  transport methods.
\newblock {\em Proc. 5th Euro. Congress of Math. Amsterdam}, 2008.

\bibitem{chertock2018high}
A.~Chertock, Y.~Epshteyn, H.~Hu, and A.~Kurganov.
\newblock High-order positivity-preserving hybrid
  finite-volume-finite-difference methods for chemotaxis systems.
\newblock {\em Advances in Computational Mathematics}, 44(1):327--350, 2018.

\bibitem{ConcaEspejoVilches11}
C.~Conca, E.~Espejo, and K.~Vilches.
\newblock Remarks on the blowup and global existence for a two species
  chemotactic keller-segel system in r2.
\newblock {\em European J. Appl. Math., 22}, 2011.

\bibitem{corrias2014existence}
L.~Corrias, M.~Escobedo, and J.~Matos.
\newblock Existence, uniqueness and asymptotic behavior of the solutions to the
  fully parabolic {Keller-Segel} system in the plane.
\newblock {\em Journal of Differential Equations Volume 257, Issue 6}, pages
  1840--1878, 2014.

\bibitem{CorriasPerthame08}
L.~Corrias and B.~Perthame.
\newblock Asymptotic decay for the solutions of the parabolic-parabolic
  {Keller-Segel} chemotaxis system in critical spaces.
\newblock {\em Mathematical and Computer Modelling, 47(7)}, pages 755--764,
  2008.

\bibitem{CTGFD2017}
K.~Z. Coytea, H.~Tabuteaue, E.~A. Gaffneyb, K.~R. Fostera, and W.~M. Durhama.
\newblock Microbial competition in porous environments can select against rapid
  biofilm growth.
\newblock {\em Proc. Natl. Acad. Sci. USA, vol. 114 no. 2, E161–E170, doi:
  10.1073/pnas.1525228113}, 2017.

\bibitem{EganaMischler16}
G.~Egana and S.~Mischler.
\newblock Uniqueness and long time asymptotic for the {Keller-Segel} equation:
  the parabolic-elliptic case.
\newblock {\em Arch. Ration. Mech. Anal. 220 (2016), no. 3}, pages 1159--1194.

\bibitem{EspejoVilchesConca13}
E.~E. Espejo, K.~Vilches, and C.~Conca.
\newblock Sharp condition for blow-up and global existence in a two species
  chemotactic keller-segel system in $\mathbb{R}^2$.
\newblock {\em European J. Appl. Math., 24}, 2013.

\bibitem{FasanoManciniPrimicerio04}
A.~Fasano, A.~Mancini, and M.~Primicerio.
\newblock Equilibrium of two populations subject to chemotaxis.
\newblock {\em Math. Models Methods Appl. Sci., 14}, 2004.

\bibitem{HeThesis}
S.~He.
\newblock Mixing, flocking and cooperation - analytical studies of transport
  phenomena in biology.
\newblock {\em Ph.D. Thesis, University of Maryland, College Park}, June 2018.

\bibitem{HillenPainter09}
T.~Hillen and K.~Painter.
\newblock A users guide to pde models for chemotaxis.
\newblock {\em Journal of mathematical biology, 58(1-2)}, pages 183--217, 2009.

\bibitem{Hortsmann}
D.~Horstmann.
\newblock {From 1970 until present: the Keller-Segel model in chemotaxis and
  its consequences}.
\newblock {\em I, Jahresber. Deutsch. Math.-Verein}, 105(3):103--165, 2003.

\bibitem{JagerLuckhaus92}
W.~J\"ager and S.~Luckhaus.
\newblock On explosions of solutions to a system of partial differential
  equations modelling chemotaxis.
\newblock {\em Trans. Amer. Math. Soc.}, 329(2):819--824, 1992.

\bibitem{KozonoSugiyama09}
H.~Kozono and Y.~Sugiyama.
\newblock Global strong solution to the semi-linear {Keller-Segel} system of
  parabolic-parabolic type with small data in scale invariant spaces.
\newblock {\em Journal of Differential Equations Volume 247, Issue 1}, pages
  1--32, 2009.

\bibitem{KurganovMedvidova14}
A.~Kurganov and M.~Lukacova-Medvidova.
\newblock Numerical study of two-species chemotaxis models.
\newblock {\em Discrete and Continuous Dynamical Systems. Series B., 19}, 2014.

\bibitem{Nagai95}
T.~Nagai.
\newblock Blow-up of radially symmetric solutions to a chemotaxis system.
\newblock {\em Adv. Math. Sci. Appl.}, 5(2):581--601, 1995.

\bibitem{NagaiSenbaYoshida97}
T.~Nagai, T.~Senba, and K.~Yoshida.
\newblock Application of the {Trudinger-Moser} inequality to a parabolic system
  of chemotaxis.
\newblock {\em Funkcialaj Ekvacioj, 40}, pages 411--433, 1997.

\bibitem{Naito06}
Y.~Naito.
\newblock Asymptotically self-similar solutions for the parabolic system
  modelling chemotaxis.
\newblock {\em Self-similar solutions of nonlinear PDE, Banach Center
  Publications, Institute of mathematics, Polish academy of sciences, Warszawa,
  74}, pages 149--160.

\bibitem{SW05}
I.~Shafrir and G.~Wolansky.
\newblock {Moser-Trudinger} and logarithmic {HLS} inequalities for systems.
\newblock {\em Journal of the European Mathematical Society}, 7, 2005.

\bibitem{TaoWinkler12}
Y.~Tao and M.~Winkler.
\newblock Boundedness in a quasilinear parabolic-parabolic {Keller-Segel}
  system with subcritical sensitivity.
\newblock {\em Journal of Differential Equations Volume 252, Issue 1}, pages
  692--715, 2012.

\bibitem{Wolansky02}
G.~Wolansky.
\newblock Multi-components chemotactic system in the absence of conflicts.
\newblock {\em European Journal of Applied Mathematics, Volume 13, Issue 6},
  2002.

\end{thebibliography}

\end{document}